\documentclass[11pt]{article}

\usepackage{amsfonts,amsthm,amsmath,amssymb}
\usepackage{graphicx}
\usepackage{mathtools,booktabs,bm,bbm,stmaryrd}
\usepackage{color}

\newcommand{\sq}{\varepsilon}
\newcommand{\ssq}{\sqrt{\varepsilon}}

\newcommand{\uN}{U}
\newcommand{\pN}{P}
\newcommand{\qN}{Q}
\newcommand{\wN}{\bm{W}}

\newcommand{\vph}{\bm{\mathbbm{z}}}
\newcommand{\vphu}{\mathbbm{v}}
\newcommand{\vphq}{\mathbbm{r}}
\newcommand{\vphp}{\mathbbm{s}}
\newcommand{\dual}[2]{\left\langle#1,#2\right\rangle}
\newcommand{\norm}   [1] {\left\Vert#1\right\Vert}
\newcommand{\enormsharp}   [1] {\interleave #1\interleave_{\sharp}}
\newcommand{\enorm}   [1] {\interleave #1\interleave_{E}}
\newcommand{\lnorm}   [1] {\interleave #1\interleave_{2}}

\newcommand{\jump}   [1] {[\![#1]\!]}
\newcommand{\af}{a_\textsf{1}}
\newcommand{\as}{a_\textsf{2}}
\newcommand{\prou}{\Pi^-}
\newcommand{\prop}{\Pi_x^+}
\newcommand{\proq}{\Pi_y^+}
\newcommand{\sobv}{\mathcal{V}_N}
\newcommand{\Bln}[2]{B(#1;#2)}
\newcommand{\spc}{\mathcal{V}_N}

\newtheorem  {proposition} {\hspace{15pt}Proposition}[section]
\newtheorem  {theorem}    {\hspace{15pt} Theorem}[section]
\newtheorem  {remark}     {\hspace{15pt} Remark}[section]

\newtheorem  {lemma}     {\hspace{15pt} Lemma}[section]
\newtheorem  {assumption}     {\hspace{15pt} Assumption}[section]

\newtheorem  {example}     {\hspace{15pt} Example}[section]

\title{\textbf{Supercloseness of the local discontinuous Galerkin method
	for a singularly perturbed convection-diffusion problem}
}

\author{
Yao Cheng\footnotemark[1]
\quad
Shan Jiang\footnotemark[2]
\quad
Martin Stynes\footnotemark[3]
}

\begin{document}

\maketitle

\footnotetext[1]
{School of Mathematical Sciences,
Suzhou University of Science and Technology,
Suzhou 215009, Jiangsu Province, China;
Research supported by NSFC grant 11801396, Natural Science Foundation of Jiangsu Province grant BK20170374, Natural Science Foundation of the Jiangsu Higher Education Institutions of China grant 17KJB110016;
\textbf{ycheng@usts.edu.cn} 
}
\footnotetext[2]
{School of Science, Nantong University, Nantong 226019, Jiangsu Province, China; Research supported by NSFC grant 11771224; \textbf{jiangshan@ntu.edu.cn}
}

\footnotetext[3]
{Corresponding author. Applied and Computational Mathematics Division, Beijing Computational Science Research Center, 
	Beijing 100193, China; Research supported by NSFC grants 12171025 and NSAF-U1930402; \textbf{m.stynes@csrc.ac.cn}
}


\begin{abstract}
A singularly perturbed convection-diffusion problem posed on the unit square in~$\mathbb{R}^2$, whose solution has exponential boundary layers, is solved numerically using the local discontinuous Galerkin (LDG) method with piecewise polynomials of degree at most~$k>0$ on three families of layer-adapted meshes: Shishkin-type, Bakhvalov-Shishkin-type and Bakhvalov-type.
On Shishkin-type meshes
this method is known to be  no greater than $O(N^{-(k+1/2)})$ accurate in the energy norm induced by the bilinear form of the weak formulation,  where $N$ mesh intervals are used in each coordinate direction. (Note: all bounds in this abstract are uniform in the singular perturbation parameter and neglect logarithmic factors that will appear in our detailed analysis.)
A delicate argument is used in this paper to establish $O(N^{-(k+1)})$ energy-norm superconvergence on all three types of mesh for the difference between the LDG solution and a local Gauss-Radau projection of the exact solution into the finite element space.
This supercloseness property implies a new $N^{-(k+1)}$ bound for the $L^2$ error between the LDG solution on each type of mesh and the exact solution of the problem; this bound is optimal (up to logarithmic factors).
Numerical experiments confirm our theoretical results.
\end{abstract}



\section{Introduction}
\label{intro}

Consider the singularly perturbed convection-diffusion problem
\begin{subequations}\label{cd:spp:2d}
	\begin{align}
	-\varepsilon \Delta u + \textbf{a} \cdot \nabla u + b u
	&= f \  \text{ in }\Omega=(0,1)\times (0,1),   \label{spp:pde}
	\\
	u &= 0\ \text{ on } \partial\Omega,  \label{spp:bc}
	\end{align}
\end{subequations}
where $\varepsilon>0$ is a small parameter, 
$\textbf{a}(x,y)=(\af(x,y),\as(x,y))\geq(\alpha_1,\alpha_2)>(0,0)$,
and 
\begin{equation}\label{assumption:coef:2d}
b(x,y)- \frac12 \nabla\cdot \textbf{a}(x,y) \geq \beta>0,
\end{equation}
for any $(x,y)\in \overline{\Omega}$.
Here $\alpha_1,\alpha_2$ and $\beta$ are some positive constants.
We assume that $\textbf{a}$, $b$ and $f$ are sufficiently smooth.
With these assumptions, it is straightforward to use the Lax-Milgram lemma to show that \eqref{cd:spp:2d}
has  a unique weak solution in $H_0^1(\Omega)\cap H_2(\Omega)$.
Note that when $\varepsilon>0$ is sufficiently small, 
condition (\ref{assumption:coef:2d}) can always be ensured
by a simple transformation
$u(x,y)=e^{t(x+y)}v$ with a suitably chosen positive constant~$t$ such that
\[
-2\varepsilon t^2+(\af+\as)t+ b- \frac12 \nabla\cdot \textbf{a} \geq \beta.
\]

The problem \eqref{cd:spp:2d} has been widely studied since it is a basic model for several applications such as the linearised Navier-Stokes equations  at high Reynolds number~\cite{Roos2008}.  
Its solution $u$ usually exhibits a boundary layer, i.e., although $u$ is bounded, it can have large derivatives in thin layer regions near certain parts of the boundary~$\partial\Omega$; see Proposition~\ref{proposition:reg:2d} below.

When solving \eqref{cd:spp:2d} numerically, the smallness of the coefficient $\varepsilon$ of the diffusion term in~\eqref{spp:pde} reduces the stability of  standard methods,  
so computed solutions may exhibit severe numerical oscillations unless one modifies the method and/or the mesh to exclude such behaviour. 
Consequently many special numerical methods have been developed to compute accurate solutions of~\eqref{cd:spp:2d}; see \cite{JKN18,Linss10,Miller2012,Roos2008,StySty18}  and their references.

One such method is the local discontinuous Galerkin (LDG) finite element method, which is stabilised across the element interfaces \cite{Cockburn:Shu:LDG}. 
The LDG method has several desirable properties 
such as strong stability,
high order accuracy, flexibility of $h–p$ adaptivity
and local solvability. 
Thus it is suited to problems whose solutions have layers or large gradients.

In recent years, the LDG method has been used to solve singularly perturbed problems such as~\eqref{cd:spp:2d} whose solutions exhibit  boundary layers; see \cite{CYWL22} and its references. For this type of problem,
numerical results show that the LDG solution computed on a uniform mesh does not oscillate~\cite{CYWL22,Xie2009JCM}. 
When the LDG method is used to solve~\eqref{cd:spp:2d} on a Shishkin mesh (S-mesh), then~\cite{Zhu:2dMC} the error of the computed solution, measured in the energy norm induced by the bilinear form of the weak formulation,
converges with rate $(N^{-1}\ln N)^{k+1/2}$, uniformly in the singular perturbation parameter~$\varepsilon$, when tensor-product piecewise polynomials of degree at most $k$ are used; here $N$~is the number of elements  in each coordinate direction. This result is generalised in~\cite{Cheng2021:Calcolo,Cheng2020} to other layer-adapted meshes including the Bakhvalov-Shishkin  mesh (BS-mesh) and Bakhvalov-type mesh (B-mesh), 
using the detailed mesh structure and properties of the local or generalized Gauss-Radau projection to prove convergence of order $(N^{-1}\max|\psi^{\prime}|)^{k+1/2}$ in the energy norm, where $\psi$ is the mesh-characterizing function which will be defined in Section~\ref{subsec:layer:adapted:meshes}.

Superconvergence of the LDG solution is considered in~\cite{Xie2010MC,Xie2009JCM,Zhu:2018}, where it is shown that for the one-dimensional analogue of~\eqref{cd:spp:2d}, the computed solution attains nodal superconvergence of 
order $(N^{-1}\ln N)^{2k+1}$ on the S-mesh. In addition, 
a convergence rate of order $(N^{-1}\max|\psi^{\prime}|)^{k+1}$ in the $L^2$ error was observed numerically in~\cite{Cheng2021:Calcolo}
for the S-type, BS-type and B-type layer-adapted meshes, but for \eqref{cd:spp:2d} the only theoretical result for the  $L^2$ error is a suboptimal convergence rate of order $(N^{-1}\max|\psi^{\prime}|)^{k+1/2}$, which follows trivially from the energy-norm convergence results of~\cite{Cheng2021:Calcolo,Cheng2020,Zhu:2dMC}.

No energy-norm superconvergence result has been proved for the LDG method applied to~\eqref{cd:spp:2d} on any layer-adapted mesh. Our paper will fill this gap in the LDG theory by considering the difference between the numerical solution and the local Gauss-Radau projection of the exact solution and proving energy-norm superconvergence of order 
$(N^{-1}\max|\psi^{\prime}|)^{k+1}$ for the three layer-adapted meshes of \cite{Cheng2020}, uniformly in~$\varepsilon$ for the S-mesh and BS-mesh and almost uniform in~$\varepsilon$ for the B-mesh; see Theorem~\ref{thm:superconvergent}.
This type of superconvergence is often described as \emph{supercloseness}; see \cite[p.395]{Roos2008}. The convergence rate is one half-order higher than the energy norm error between the computed and true solutions. It implies the new result that the $L^2$ error between the numerical and true solutions on these three meshes has the rate $(N^{-1}\max|\psi^{\prime}|)^{k+1}$, which we will show to be sharp in numerical experiments. 

Our approach is different from \cite{Zhu:2dMC}, where a streamline-diffusion-type norm
was used in the analysis of the convective error 
to absorb the derivative 
and the jump of the error in the finite element space.
It does not seem possible
to prove superconvergence using this stronger norm 
because the underlying inverse estimate
leads to a suboptimal bound in the analysis,  
so we turn to a different strategy: control of the derivative 
and the jump of the error in the finite element space 
by the energy-norm itself  plus a term with optimal convergence rate
(see Lemma~\ref{relation-to-auxiliary}).
Although there is still a negative power of the small parameter in this upper bound, we are nevertheless able to improve the error estimate for the convection term by using the small measure of the layer region; see~\eqref{T43} below.
We also have to handle the error jump across the finite element interfaces
in a sharp way. These innovations, combined with some slightly improved approximation properties,  
deliver the desired supercloseness result. 

The paper is organised as follows. In Section~\ref{sec:scheme}, 
we define the layer-adapted meshes and present the LDG method. 
Section~\ref{sec:prelim} is devoted to the proofs of several delicate technical results that will be needed later.
Our main  supercloseness result is derived in Section \ref{sec:proof}.  
In Section~\ref{sec:experiments},  
we present some numerical experiments to confirm our theoretical results. 
Finally,  Section~\ref{sec:conclusion} gives some concluding remarks.

\emph{Notation.} We use $C$ to denote a generic positive constant that may depend on the data $\textbf{a}, b,f$ of~\eqref{cd:spp:2d}, the parameter $\sigma$ of~\eqref{tau}, and the degree $k$ of the polynomials in our finite element space,  but is independent of $\varepsilon$ and of $N$ (the number of mesh intervals in each coordinate direction); $C$ can take different values in different places.

 The usual Sobolev spaces $W^{m,\ell}(D)$ and $L^{\ell}(D)$ will be used, where $D$ is any measurable two-dimensional subset of~$\Omega$. The $L^{2}(D)$ norm is denoted by $\norm{\cdot}_{D}$,
the $L^{\infty}(D)$ norm by $\norm{\cdot}_{L^\infty(D)}$,
and $\dual{\cdot}{\cdot}_D$ denotes the $L^2(D)$ inner product.
The subscript $D$ will always be dropped when $D= \Omega$.

\section{Layer-adapted meshes and the LDG method}
\label{sec:scheme}

Typical solutions $u$ of \eqref{cd:spp:2d} have exponential layers along the sides $x=1$ and $y=1$ of~$\Omega$, and a corner layer at $(x,y)=(1,1)$.
Many authors assume that  the solution $u$ can be decomposed as follows  into a smooth component $S$ 
and layer components $E_{21}, E_{12}$ and $ E_{22}$; 
see, e.g., \cite[Section 2.2]{Zhu:2dMC}.

\begin{proposition}\label{proposition:reg:2d}
	Let $m$ be a non-negative integer. Let $\kappa$ satisfy $0<\kappa<1$.
	Under certain smoothness and compatibility conditions on the data,
	the problem \eqref{cd:spp:2d} has a solution 
	$u$ in the H\"older space $C^{m+2,\kappa}(\Omega)$, and this solution can be decomposed as 	$u = S +E_{21}+E_{12}+E_{22}$, where
	\begin{subequations}\label{reg:u:2d}
		\begin{align}
		\label{reg:S}
		\left| \partial_x^{i}\partial_y^{j}S(x,y)\right|
		\leq&\; C,
		\\
		\label{reg:E21}
		\left| \partial_x^{i}\partial_y^{j}E_{21}(x,y)\right|
		\leq &\; C\varepsilon^{-i} e^{-\alpha_1(1-x)/\varepsilon},
		\\
		\label{reg:E12}
		\left| \partial_x^{i}\partial_y^{j}E_{12}(x,y)\right|
		\leq &\; C\varepsilon^{-j} e^{-\alpha_2(1-y)/\varepsilon},
		\\
		\label{reg:E22}
		\left| \partial_x^{i}\partial_y^{j}E_{22}(x,y)\right|
		\leq &\; C\varepsilon^{-(i+j)}e^{-[\alpha_1(1-x)+\alpha_2(1-y)]/\varepsilon},
		\end{align}
	\end{subequations}
	for all $(x,y)\in \bar{\Omega}$ and
	all nonnegative integers $i,j$ with $i+j\le m+2$. 
\end{proposition}

This proposition is proved in \cite[p.244, Theorem 7.17]{Linss10}
	and \cite[p.253, Theorem 1.25]{Roos2008} for the case $m=1$.
	Similar arguments will work for larger values of $m$  
	provided  the data in~\eqref{cd:spp:2d} has sufficient regularity and satisfies suitable 
	compatibility conditions at the corners of~$\Omega$.
	Like  \cite[Section 2.2]{Zhu:2dMC} we shall need $m=k$ in Proposition~\ref{proposition:reg:2d}, where $k$ is the degree of the piecewise polynomials in our finite element space.

\subsection{Layer-adapted meshes}
\label{subsec:layer:adapted:meshes}

We shall use layer-adapted meshes \cite{Cheng2020,Linss10,Roos2008} that are refined near the sides $x=1$ and $y=1$ of~$\Omega$ but are uniform otherwise. These meshes are tensor products of one-dimensional meshes. 

Let $\varphi:[0, 1/2]\to [0,\infty)$ be a \emph{mesh-generating function} satisfying
$\varphi(0)=0, \varphi^{\prime}>0$ and $\varphi^{\prime\prime}\geq0$.
The point at which each mesh switches from uniform to nonuniform will be defined in~\eqref{layer-adapted:mesh} via the \emph{mesh transition parameter}
\begin{equation}  \label{tau}
\tau := \min\left\{\frac12, 
\frac{\sigma\varepsilon}{\alpha} \varphi\left(\frac12\right)\right\},
\end{equation}
where $\sigma>0$ is a user-chosen parameter whose value affects our error estimates; 
in the supercloseness analysis
below it can be seen that~$\sigma$ needs to be sufficiently large. 
In~\eqref{tau} we have assumed that $\alpha_1=\alpha_2=:\alpha$ for notational simplicity; the case $\alpha_1\ne \alpha_2$ does not introduce any additional analytical difficulties but would complicate our notation. 

Let $N\geq 4$ be an even positive integer. Each of our meshes will use $N+1$ points in each coordinate direction.

During the rest of the paper, we make the following two assumptions to simplify some details of the analysis.
\begin{assumption}\label{ass:1}
\mbox{ }
\begin{itemize}
\item[(i)]
Assume in~\eqref{tau}  that 
$\tau =(\sigma\varepsilon/\alpha)\varphi\big(1/2\big)$.
\item[(ii)]
Assume that $\sq\leq N^{-1}$.
\end{itemize}
\end{assumption}
If Assumption~\ref{ass:1}(i) is not satisfied, then the problem \eqref{cd:spp:2d} is not singularly perturbed and can be analysed in a classical diffusion-dominated framework.

Assumption~\ref{ass:1}(ii) is assumed by most authors 
who study numerical methods for solving~\eqref{cd:spp:2d}
--- e.g., in \cite{Roos2003,Zarin2014,Zhu:2018} it is used in the analyses of DG methods.
In Remark~\ref{comments:ep:N} we discuss the effect on our analysis of removing Assumption~\ref{ass:1}(ii).

Define the mesh points $(x_i,y_j)$ for $i,j=0,1,\dots, N$ by
\begin{equation}\label{layer-adapted:mesh}
x_i=y_i=
\begin{dcases}
2(1-\tau)i/N   &\text{for } i=0,1,...,N/2,
\\
	1-\frac{\sigma\varepsilon}{\alpha}\varphi\left(\frac{N-i}{N}\right)
&\text{for } i=N/2+1,N/2+2,...,N.
\end{dcases}
\end{equation}
Associated with each $\varphi$ is the \emph{mesh-characterizing function} $\psi:=e^{-\varphi}$,
which will play an important role in our convergence analysis. 

As in \cite{Cheng2020,Linss10,Roos2008} we consider three types of layer-adapted mesh:
the Shishkin mesh (S-mesh), the Bakhvalov-Shishkin mesh (BS-mesh) and the Bakhvalov-type mesh (B-mesh).
Table~\ref{table:functions} lists the main properties of $\varphi$ and $\psi$ for these meshes.

\begin{table}[ht]
	\caption{Three layer-adapted meshes.}
	\label{table:functions}
	\normalsize
	\centering
	\begin{tabular}{ccccc}
		\toprule
		& S-mesh
		& BS-mesh
		& B-mesh
		\\
		\midrule
		$\varphi(t)$             & $2t\ln N$    & $-\ln\big[1-2(1- N^{-1})t\big]$  &$-\ln\big[1-2(1-\varepsilon)t\big]$   \\
		$\varphi(1/2)$           & $\ln N$      & $\ln N$            &$\ln (1/\varepsilon)$    \\
		$\min\varphi^{\prime}$    & $2\ln N$     & $2$                & $2$ \\
		$\max \varphi^{\prime}$ & $2\ln N$     & $2N$               & $2\varepsilon^{-1}$\\
		$\psi(t)$                & $N^{-2t}$    & $1-2(1-N^{-1})t$   &$1-2(1-\varepsilon)t$ \\
		$\psi(1/2)$              & $N^{-1}$     & $ N^{-1}$          & $\varepsilon$    \\
		$\max|\psi^{\prime}|$    & $2\ln N$     & $2$                & $2$
		\\
		\bottomrule
	\end{tabular}
\end{table}

Finally,  by drawing axiparallel lines through the mesh points $(x_i,y_j)$,
we construct the  layer-adapted mesh:
set $\Omega_N:=\{K_{ij}\}_{i, j=1,\dots,N}$,
where each rectangular mesh element 
$K_{ij} :=I_i\times J_j :=(x_{i-1},x_i)\times (y_{j-1},y_j)$.
Figure \ref{fig:division} displays these meshes for 
$\sigma = 4,\varepsilon = 10^{-2},\alpha=1$ 
and $N=8$ in \eqref{tau} and~\eqref{layer-adapted:mesh};
they are uniform and coarse on $\Omega_{11}:=(0,1-\tau)\times(0,1-\tau)$, 
but are refined perpendicularly to the sides $x=1$ and $y=1$ in the regions $\Omega_x:=\Omega_{21}\cup \Omega_{22}$
and $\Omega_y:=\Omega_{12}\cup \Omega_{22}$ respectively.

We set $h_i = x_i-x_{i-1} = y_i-y_{i-1}$ for $i=1,2,\dots, N$. For all three types of mesh one has $N^{-1}\le h_i\le 2N^{-1}$ for $0< i\le N/2$, and when $N/2<i\le N$, for some constant~$C$ one has
\begin{equation}
\label{mesh:size:layer}
C\varepsilon N^{-1}\min\varphi^{\prime}
\leq h_i \leq 
C\varepsilon N^{-1}\max\varphi^{\prime}.
\end{equation}
When $N/2<i\le N$, the bound \eqref{mesh:size:layer} yields  $h_i \le C\varepsilon N^{-1}\ln N$ for the S-mesh,
$h_i \le C\varepsilon$ for the BS-mesh and $h_i \le C N^{-1}$ for the B-mesh. 

For all three of our meshes, Assumption~\ref{ass:1}(ii) implies that $\psi(1/2)\leq N^{-1}$ and $h_i\leq CN^{-1}$ for $i=1,2,\dots,N$; 
these properties are used in our analysis.

\begin{figure}[h]
	\begin{minipage}{0.49\linewidth}
		\vspace{3pt}
		\centerline{\includegraphics[width=1.3\textwidth]{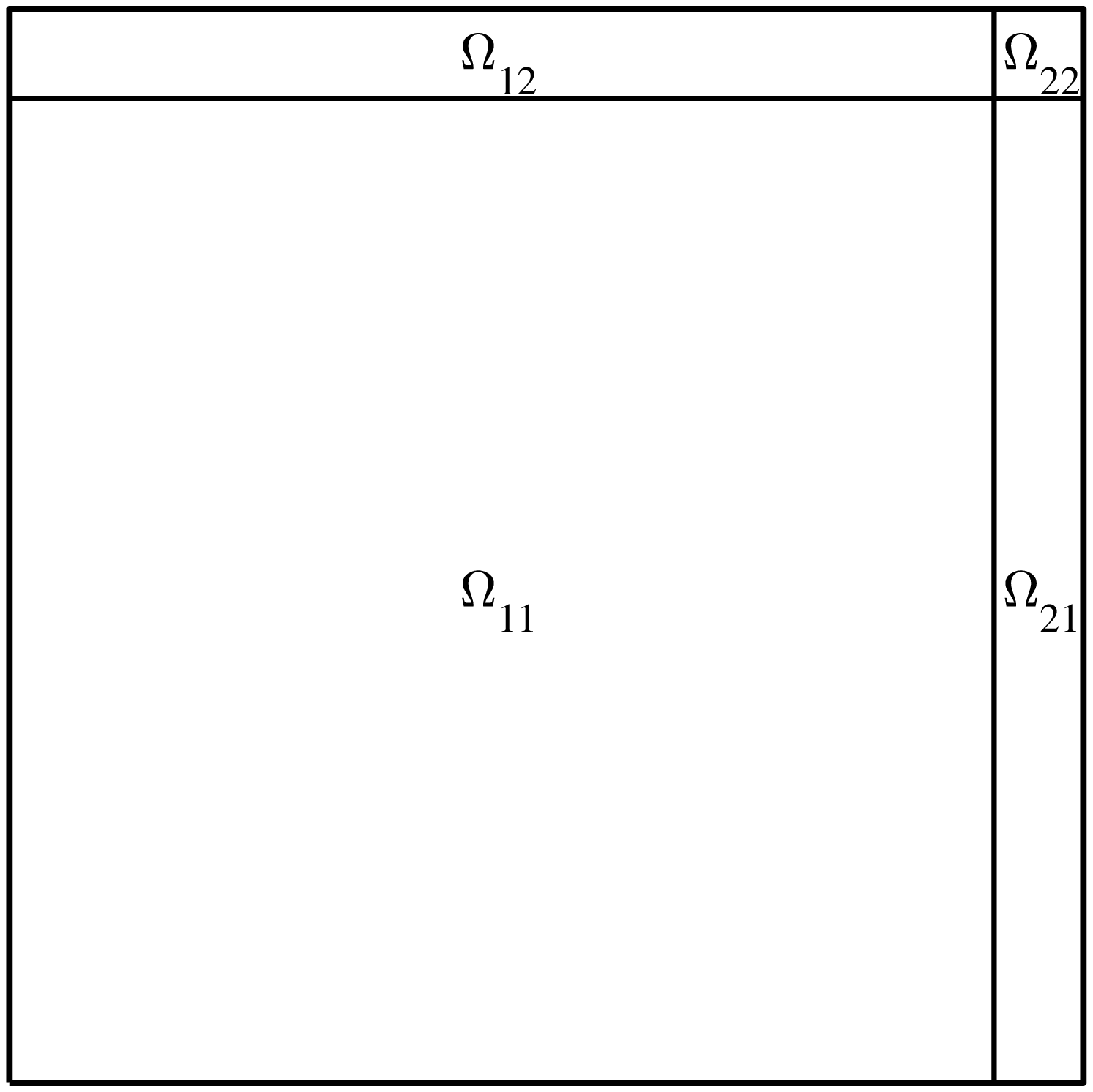}}
		\centerline{Subregions of $\Omega$}
		\vspace{3pt}
		\centerline{\includegraphics[width=1.3\textwidth]{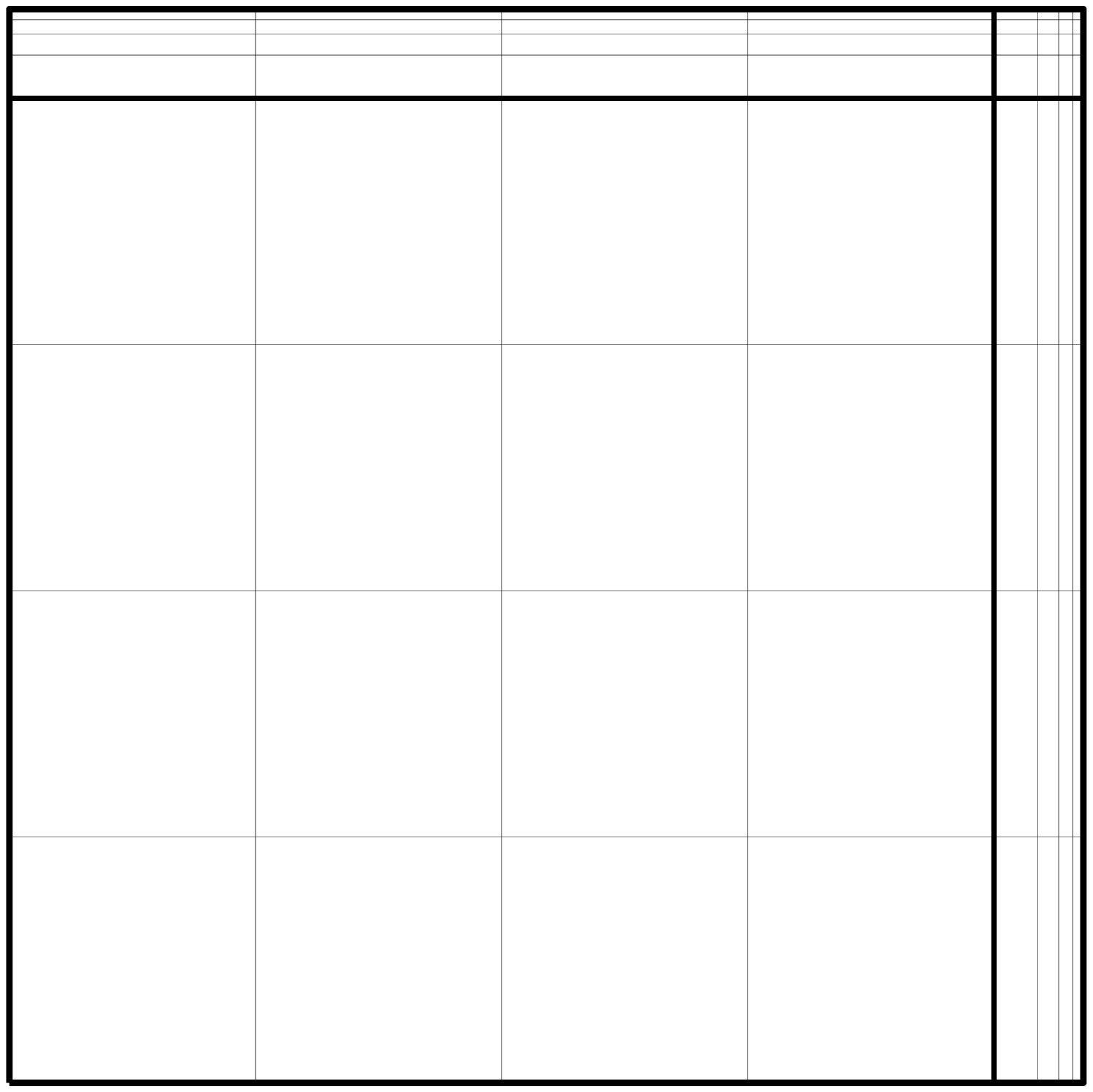}}
		\centerline{BS-mesh}
	\end{minipage}
    \begin{minipage}{0.49\linewidth}
		\vspace{3pt}
		\centerline{\includegraphics[width=1.3\textwidth]{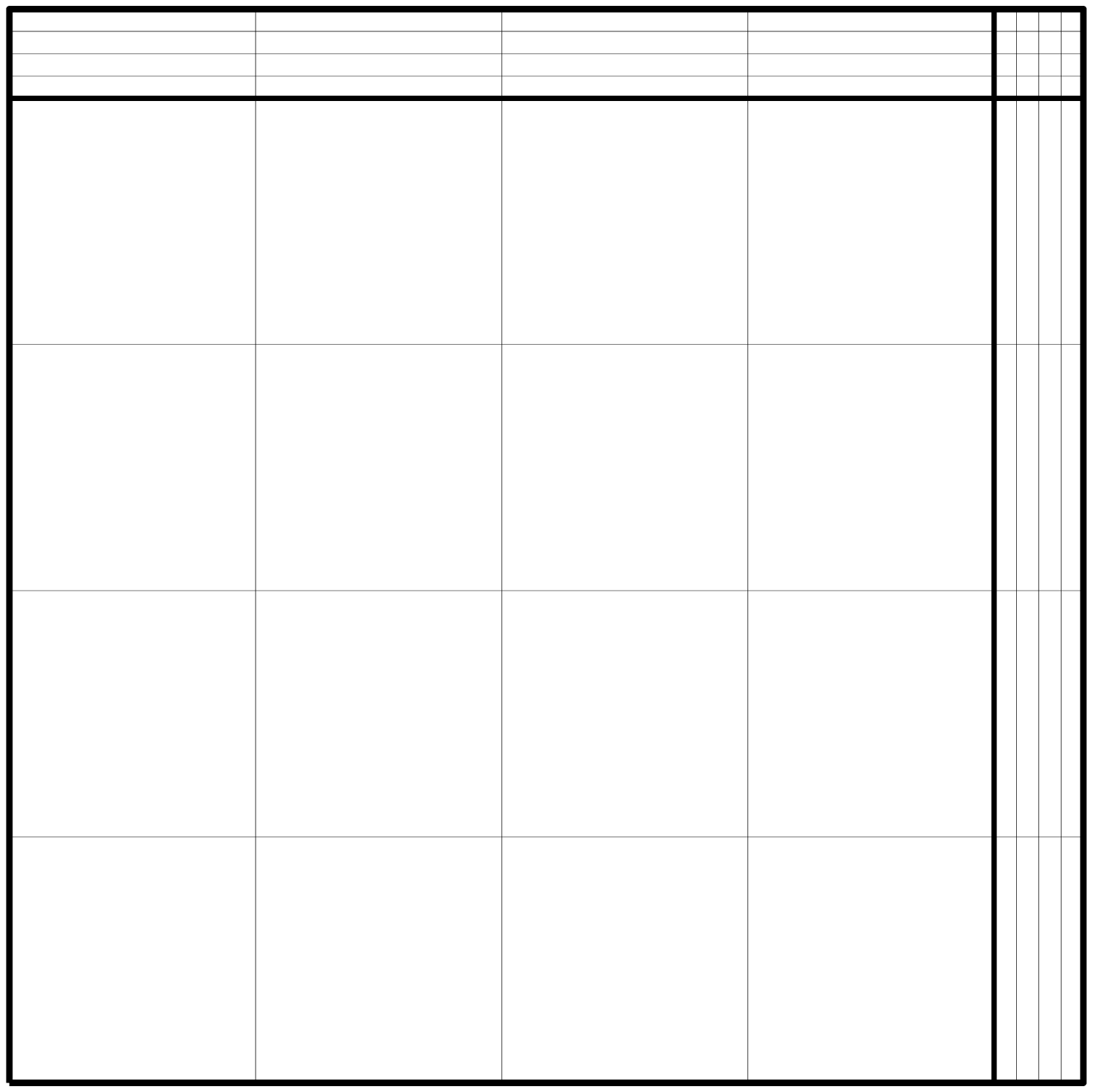}}
		\centerline{S-mesh}
		\vspace{3pt}
		\centerline{\includegraphics[width=1.3\textwidth]{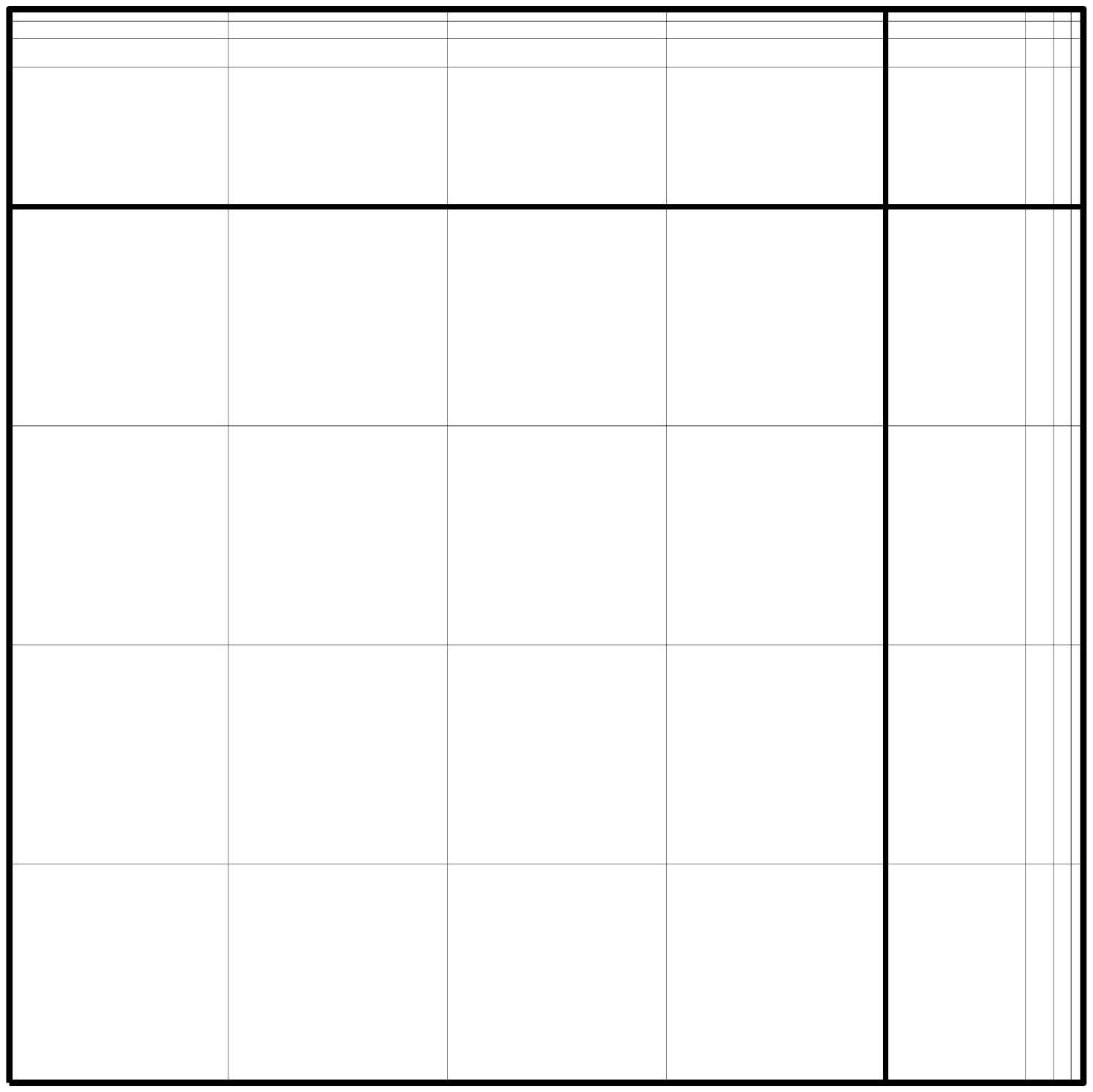}}
		\centerline{B-mesh}
	\end{minipage}
	\caption{Domain division and three layer-adapted meshes with $N=8$.}
	\label{fig:division}
\end{figure}

\subsection{The local discontinuous Galerkin (LDG) method}\label{sec:LDG}

Let $k$ be a fixed positive integer. On any 1-dimensional interval~$I$, let  $\mathcal{P}^{k}(I)$ denote the space of polynomials of degree at most~$k$ defined on~$I$.  For each mesh element $K =  I_i \times J_j$, set $\mathcal{Q}^k(K) := \mathcal{P}^{k}(I_i)\otimes \mathcal{P}^{k}(J_j)$.
Then define the discontinuous finite element space  
\[
\spc= \left\{v\in L^2 (\Omega)\colon
v|_{K} \in \mathcal{Q}^{k} (K),  K\in\Omega_N\right\}.
\]
Note that functions in $\spc$ are allowed to 
be discontinuous across element interfaces.
For any $v\in \spc$ and $y\in J_j$, $j=1,2,\dots,N$,
we use $v^\pm_{i,y}=\lim_{x\to x_{i}^\pm}v(x,y)$
to express the traces on element edges.
The jumps on the vertical edges are denoted by
$\jump{v}_{i,y}:= v^{+}_{i,y}-v^{-}_{i,y}$
for $i=1,2,\dots,N-1$, $\jump{v}_{0,y}:=v^{+}_{0,y}$
and $\jump{v}_{N,y}:=-v^{-}_{N,y}$.
In a similar fashion, we can define the jumps
$\jump{v}_{x,j}$, $j=0,1,\dots,N$.

To define the LDG method, rewrite \eqref{cd:spp:2d} as an  equivalent first-order system:
\[
-p_x-q_y+\af u_x+\as u_y +bu=f,\quad p=\sq u_x,\quad q=\sq u_y
\]
with the homogeneous boundary condition of~\eqref{spp:bc}.
Then apply the original DG discretization \cite{Reed1973} to this system. 
In the definition of the numerical flux, 
we employ a purely upwind flux \cite{Cockburn:Shu:2001} 
for the convection term and a purely alternating numerical flux \cite{Cockburn2001} 
for the diffusion term.
The final compact form of the LDG method reads as follows \cite{Cheng2020}:\\
\indent Find $\wN=(\uN,\pN,\qN)\in \spc^3 := \spc\times\spc\times\spc$ 
($U$ approximates~$u$, while $\pN$ and $\qN$ approximate $p$ and $q$  respectively)
such that
\begin{equation}\label{compact:form:2d}
	\Bln{\wN}{\vph}=\dual{f}{\vphu}
	\quad \forall \vph=(\vphu,\vphp,\vphq)\in \spc^3,
\end{equation}
where 
\begin{align}
\label{B:def:2d}
B(\wN;\vph) &:=
\mathcal{T}_1(\wN;\vph)+\mathcal{T}_2(\wN;\vph)
+\mathcal{T}_3(\wN;\vph)+\mathcal{T}_4(\uN;\vphu),
\end{align}
with
\begin{align}
\mathcal{T}_1(\wN;\vph)&=
\varepsilon^{-1}[\dual{\pN}{\vphp}+\dual{\qN}{\vphq}]
+\dual{(b- \nabla\cdot \textbf{a})\uN}{\vphu},
\nonumber\\
\mathcal{T}_2(\wN;\vph)&=
\dual{ \uN}{\vphp_x}
+\sum_{j=1}^{N}\sum_{i=1}^{N-1}\dual{\uN^{-}_{i,y}}{\jump{\vphp}_{i,y}}_{J_j}
+\dual{ \uN}{\vphq_y}
+\sum_{i=1}^{N}\sum_{j=1}^{N-1}\dual{\uN^{-}_{x,j}}{\jump{\vphq}_{x,j}}_{I_i},
\nonumber\\
\mathcal{T}_3(\wN;\vph)&=
\dual{\pN}{\vphu_x}
+\sum_{j=1}^{N}\Big[
\sum_{i=0}^{N-1}\dual{\pN^{+}_{i,y}}{\jump{\vphu}_{i,y}}_{J_j}
-\dual{\pN^{-}_{N,y}}{\vphu^{-}_{N,y}}_{J_j}
\Big]
\nonumber\\
&\quad+\dual{\qN}{\vphu_y}
+\sum_{i=1}^{N}\Big[
\sum_{j=0}^{N-1}\dual{\qN^{+}_{x,j}}{\jump{\vphu}_{x,j}}_{I_i}
-\dual{\qN^{-}_{x,N}}{\vphu^{-}_{x,N}}_{I_i}
\Big],
\nonumber\\
\mathcal{T}_4(\uN;\vphu)&=
-\dual{\af\uN}{\vphu_x}
-\sum_{j=1}^{N}\Big[\sum_{i=1}^{N}
\dual{(\af)_{i,y}\uN^{-}_{i,y}}{\jump{\vphu}_{i,y}}_{J_j}
-\dual{\lambda_{1}\uN^{-}_{N,y}}{\vphu^{-}_{N,y}}_{J_j}
\Big]
\nonumber\\
&\quad-\dual{\as\uN}{\vphu_y}
-\sum_{i=1}^{N}\Big[\sum_{j=1}^{N}
\dual{(\as)_{x,j}\uN^{-}_{x,j}}{\jump{\vphu}_{x,j}}_{I_i}
-\dual{\lambda_{2}\uN^{-}_{x,N}}{\vphu^{-}_{x,N}}_{I_i}
\Big].
\nonumber
\end{align}
The penalty parameters $\lambda_{i}$ in the LDG method
are sometimes chosen so as to improve
the stability and accuracy of the numerical scheme \cite{Castillo2002,Cockburn:Dong:2007},
but in our paper we only require $0\leq \lambda_{i}\leq C$ as in~\cite{Zhu:2dMC}.
We shall take $\lambda_1 = \lambda_2 =0$ in the numerical experiments of Section~\ref{sec:experiments}.

Define an energy norm $\enorm{\cdot}$ on $\spc^3$  
by $\enorm{\bm{V}}^2=B(\bm{V};\bm{V})$ for each $\bm{V} = (V_u,V_p,V_q)\in \spc^3$; that is, 
\begin{align*}
\enorm{\bm{V}}^2
&=\lnorm{\bm{V}}^2
+\sum_{j=1}^{N}
\left[\sum_{i=0}^{N-1}\frac12 \dual{(\af)_{i,y}}{\jump{V_u}^2_{i,y}}_{J_j}
+\dual{\frac12(\af)_{N,y}+\lambda_{1}}{\jump{V_u}^2_{N,y}}_{J_j}
\right]
\\
&\hspace{1.5cm}
+\sum_{i=1}^{N}
\left[\sum_{j=0}^{N-1}\frac12 \dual{(\as)_{x,j}}{\jump{V_u}^2_{x,j}}_{I_i}
+\dual{\frac12(\as)_{x,N}+\lambda_{2}}{\jump{V_u}^2_{x,N}}_{I_i}
\right],
\\
\text{where } \lnorm{\bm{V}}^2
&:=
\varepsilon^{-1}\norm{V_p}^2+\varepsilon^{-1}\norm{V_q}^2 
+ \norm{ \left(b- \frac12 \nabla\cdot \textbf{a} \right)^{1/2}V_u}^2.
\end{align*}
The linear system of equations \eqref{compact:form:2d} has a unique solution  $\wN$ because the associated homogeneous problem (i.e., with $f=0$) has $\enorm{\wN} =0$ and hence $\wN = (0,0,0)$.

\section{ Preliminary results}
\label{sec:prelim}

This section contains several technical results that will be used in the proof of Theorem~\ref{thm:superconvergent}, our supercloseness result.

In Section~\ref{subsec:eta} we will define a local Gauss-Radau projector $\bm\Pi: \left(C(\bar\Omega)\right)^3\to  \spc^3$.
Set $\bm w = (u,p,q)$ and $\bm\Pi \bm w=(\prou u,\prop p,\proq q)$. Then the error in the LDG solution is 
$\bm e:=(e_u,e_p,e_q) :=(u-\uN,p-\pN,q-\qN)$,  which one can also write as
\[
\bm e = \bm w - \bm W = (\bm w-\bm\Pi \bm w)-(\wN-\bm\Pi \bm w)  = \bm \eta-\bm \xi,
\]
where we define 
\begin{equation}\label{error:decomposition}
\begin{split}
\bm \eta  &= (\eta_u,\eta_p,\eta_q)
= (u-\prou u,	p-\prop p,		q-\proq q),
\\
\bm \xi &= (\xi_u,\xi_p,\xi_q)
= (\uN-\prou u,	\pN-\prop p, \qN-\proq q)
\in \sobv^3.
\end{split}
\end{equation}

Section~\ref{subsec:eta} is devoted to 
an estimate of the error $\bm \eta$ in the projection of the exact solution, while 
in Section~\ref{subsec:xi} we shall bound  $\bm \xi$.

For the rest of our paper we make the following assumptions.

\begin{assumption}\label{ass:2}
	\mbox{ }
	\begin{itemize}
		\item[(i)]
		Assume that Proposition~\ref{proposition:reg:2d} is valid for $m=k$.
		\item[(ii)]
		Assume that $\sigma\geq k+2$ in \eqref{tau}.
	\end{itemize}
\end{assumption}

\subsection{The projection error $\bm\eta$}\label{subsec:eta}
Let $K_{ij}=I_i\times J_j=(x_{i-1},x_i)\times (y_{j-1},y_j)$ be any element of our mesh.
For any $z\in C(\bar{\Omega})$,  
the local Gauss-Radau projection $\Pi^{-}z \in \spc$ is defined
by the following conditions:
	\begin{align*}
	\int_{K_{ij}}(\Pi^{-}z) \vphu \,\textrm{d}x\,\textrm{d}y
	&= \int_{K_{ij}}z \vphu \,\textrm{d}x\,\textrm{d}y
	\quad \forall \vphu\in \mathcal{Q}^{k-1}(K_{ij}),
	\\
	\int_{J_j}(\Pi^{-} z)_{i,y}^{-}\vphu\,\textrm{d}y
	&=\int_{J_j} z_{i,y}^{-}\vphu\,\textrm{d}y
	\qquad \forall \vphu\in \mathcal{P}^{k-1}(J_j),
	\\
	\int_{I_i}(\Pi^{-} z)_{x,j}^{-}\vphu\,\textrm{d}x
	&=\int_{I_i} z_{x,j}^{-}\vphu\,\textrm{d}x
	\qquad \forall \vphu\in \mathcal{P}^{k-1}(I_i),
	\\
	(\Pi^{-} z)(x_{i}^{-},y_{j}^{-}) &=	 z(x_{i}^{-},y_{j}^{-});
   \end{align*}
where we used the edge traces $z_{i,y}^{-}$ and $z_{x,j}^{-}$
from Section~\ref{sec:LDG}.
To deal with the auxiliary variables $p$ and $q$, we define
two another projections.
For any $z\in C(\bar{\Omega})$,  $\Pi_x^{+}z \in \spc$ satisfies
\begin{align*}
\int_{K_{ij}}(\Pi_x^{+} z) \vphu
\,\textrm{d}x\,\textrm{d}y
&=
\int_{K_{ij}}z \vphu \,\textrm{d}x\,\textrm{d}y
\quad \forall \vphu\in \mathcal{P}^{k-1}(I_i)\otimes \mathcal{P}^k(J_j),
\\
\int_{J_j}(\Pi_x^{+} z)_{i,y}^{+}\vphu\,\textrm{d}y
&=
\int_{J_j}  z_{i,y}^{+}\vphu\,\textrm{d}y
\qquad \forall \vphu\in \mathcal{P}^k(J_j).
\end{align*}
 Analogously, the projection  $\Pi_y^{+}z \in \spc$ satisfies
 \begin{align*}
 \int_{K_{ij}}(\Pi_y^{+}  z) \vphu\,\textrm{d}x\,\textrm{d}y
 &=
 \int_{K_{ij}}z \vphu \,\textrm{d}x\,\textrm{d}y
 \quad \forall \vphu\in \mathcal{P}^{k}(I_i)\otimes \mathcal{P}^{k-1}(J_j),
 \\
 \int_{I_i}(\Pi_y^{+} z)^{+}_{x,j}\vphu\,\textrm{d}x
 &=
 \int_{I_i}  z^{+}_{x,j}\vphu\,\textrm{d}x
 \qquad \forall \vphu\in \mathcal{P}^k(I_i).
 \end{align*}

These conditions define $\Pi^{-}z, \Pi_x^{+}z, \Pi_y^{+}z  \in \spc$ uniquely \cite{Castillo2002,Cockburn2001}.
Let  $\Pi\in\{\Pi^{-},\Pi_x^{+},\Pi_y^{+}\}$.
Similarly to~\cite[Lemma 5]{Cheng2021:Calcolo},
one obtains the following stability properties: 
	\begin{subequations}\label{2GR:stb:app}
		\begin{align}
		\label{GR:stb}
		\norm{\Pi z}_{L^\infty(K_{ij})}
		&\leq  C \norm{z}_{L^\infty(K_{ij})},
		\\
		\norm{\Pi^+_x z}_{K_{ij}}&\leq  C
		\Big[
		\norm{z}_{K_{ij}} 
		+h^{1/2}_{i}\norm{z^+_{i-1,y}}_{J_j}
		\Big],
		\\
		\norm{\Pi^+_y z}_{K_{ij}}&\leq  C
		\Big[
		\norm{z}_{K_{ij}} + h^{1/2}_{j}\norm{z^+_{x,j-1}}_{I_i}
		\Big],
		\end{align}
	\end{subequations}
and the approximation properties 
(see, e.g., \cite[Lemma 3]{Cheng2021:Calcolo} and \cite[Lemma 4.3]{Zhu:2dMC})
\begin{subequations}\label{2GR:stb:app:2}
	\begin{align}
	\label{GR:app}
	\norm{z-\Pi z}_{L^\infty(K_{ij})}
	&\leq 
	C \left[h_{i}^{k+1}\norm{\partial_x^{k+1}z}_{L^\infty(K_{ij})}
	+h_{j}^{k+1}\norm{\partial_y^{k+1}z}_{L^\infty(K_{ij})}
	\right],
	\\
	\label{GR:L^2:app}
	\norm{z-\Pi z}_{K_{ij}}
	&\leq 
	C \left[h_{i}^{k+1}\norm{\partial_x^{k+1}z}_{K_{ij}}
	+h_{j}^{k+1}\norm{\partial_y^{k+1}z}_{K_{ij}}
	\right].
	\end{align}
\end{subequations}

Define
\[
\mu=
	\begin{cases}
	\varepsilon &\text{for the S-mesh}, \\
	\varepsilon\ln N &\text{for the BS-mesh},\\
	\varepsilon\ln(1/\varepsilon) &\text{for the B-mesh}.
	\end{cases}
\]

Recall the decomposition of~$u$ in Proposition~\ref{proposition:reg:2d}. Analogously decompose 
$\eta_u = \eta_S +\eta_{E_{21}}+\eta_{E_{12}}+ \eta_{E_{22}}$, where for each component~$z$ of~$u$ one sets
$\eta_z:= z - \Pi^-z$. Then one has the following bounds on the projection error~$\bm\eta$.

\begin{lemma}\label{lemma:GR}
There exists a constant $C$ such that
	\begin{subequations}\label{2dGR:property}
		\begin{align}
		\label{etau:2d}
        \norm{\eta_u}_{L^{\infty}(\Omega)}
        &\leq C(N^{-1}\max|\psi^{\prime}|)^{k+1},
		\\
		\label{etau:layer:L2}
		\norm{\eta_u}_{\Omega_{x}\cup \Omega_{y}}
		&\leq 
		C\mu^{1/2}(N^{-1}\max|\psi^{\prime}|)^{k+1},
		\\
		\label{etau:2d:1}
		\sum_{j=1}^{N}\norm{(\eta_u)_{i,y}^{-}}^2_{J_j}
		+\sum_{i=1}^{N}\norm{(\eta_u)_{x,j}^{-}}^2_{I_i}
		&\leq  CN^{-2(k+1)}(\max|\psi^{\prime}|)^{2k+1},
		\text{ for } i,j=1,...,N,
		\\
		\label{etau:2d:2}
		\sum_{i=1}^{N}\sum_{j=1}^{N}
		\Big[
		\norm{(\eta_u)_{i,y}^{-}}_{J_j}^2
		+\norm{(\eta_u)_{x,j}^{-}}_{I_i}^2\Big]
		&\leq  C(N^{-1}\max|\psi^{\prime}|)^{2k+1},
		\\
		\label{etau:2d:3}
		\sum_{j=1}^{N}\sum_{i=0}^{N}
		\dual{1}{\jump{\eta_u}^2_{i,y}}_{J_j}
		+\sum_{i=1}^{N}\sum_{j=0}^{N}
		\dual{1}{\jump{\eta_u}^2_{x,j}}_{I_i}
		&\leq  C(N^{-1}\max|\psi^{\prime}|)^{2k+1},
		\\
		\label{etap:2d:1}
		\varepsilon^{-\frac12}\norm{\eta_q}
		+\varepsilon^{-\frac12}\norm{\eta_p}
		&\leq  C(N^{-1}\max|\psi^{\prime}|)^{k+1},
		\\
		\label{etap:2d:2}
		\sum_{i=1}^{N}\norm{(\eta_q)_{x,N}^{-}}^2_{I_i}
		+\sum_{j=1}^{N}\norm{(\eta_p)_{N,y}^{-}}^2_{J_j}
		&\leq  C(N^{-1}\max|\psi^{\prime}|)^{2(k+1)}.
		\end{align}
	\end{subequations}
\end{lemma}

\begin{proof}	
The inequalities \eqref{etau:2d} and \eqref{etau:2d:2}--\eqref{etap:2d:2}  are derived in  \cite[ Lemma 4.1]{Cheng2020}.  We shall now prove \eqref{etau:layer:L2} and \eqref{etau:2d:1}.

We have $\norm{\eta_u}_{\Omega_{x}\cup \Omega_{y}} \le \sqrt{2\tau}\norm{\eta_u}_{L^\infty(\Omega_{x}\cup \Omega_{y})}$ 	because the area of $\Omega_{x}\cup \Omega_{y}$ is bounded by~$2\tau$;  thus \eqref{etau:layer:L2} follows immediately from \eqref{etau:2d} 
for the BS-mesh and B-mesh.
We now prove  \eqref{etau:layer:L2}  for the S-mesh.
Use  \eqref{reg:S} and~\eqref{GR:app} to get $\norm{\eta_{S}}_{L^\infty(\Omega_{x})}\leq C N^{-(k+1)}$; 
then $\norm{\eta_{S}}_{\Omega_{x}}\leq C\ssq N^{-(k+1)}\ln^{1/2} N$ follows.
Next, the $L^{2}$-approximation property
	\eqref{GR:L^2:app} and the derivative bound \eqref{reg:E21} yield
\begin{align*}
\norm{\eta_{E_{21}}}_{\Omega_{x}}^2
&\leq  
C\sum_{j=1}^N\sum_{i=N/2+1}^N
\Big[h_i^{2(k+1)}
\norm{\partial_x^{k+1}E_{21}}^2_{K_{ij}}
+h_{j}^{2(k+1)}\norm{\partial_y^{k+1}E_{21}}^2_{K_{ij}}
\Big]
\\
&\leq 
C\sum_{j=1}^N\sum_{i=N/2+1}^N
(N^{-1}\ln N)^{2(k+1)}
\norm{e^{-\alpha_1(1-x)/\varepsilon}}^2_{K_{ij}}
\leq C\varepsilon(N^{-1}\ln N)^{2(k+1)}.
\end{align*}
Analogously, one has 
$\norm{\eta_{E_{12}}}_{\Omega_{y}}^2
\leq  C\varepsilon(N^{-1}\ln N)^{2(k+1)}$.
Now \eqref{GR:stb} and $\sigma\geq k+1$ give
\begin{align*}
\norm{\eta_{E_{12}}}_{\Omega_{x}}^2
&= \norm{\eta_{E_{12}}}_{\Omega_{22}}^2+\norm{\eta_{E_{12}}}_{\Omega_{21}}^2
\\
&\leq  \norm{\eta_{E_{12}}}_{\Omega_{y}}^2
+C\sum_{j=1}^{N/2}\sum_{i=N/2+1}^N h_ih_j
\norm{e^{-\alpha_2(1-y)/\varepsilon}}^2_{L^\infty(K_{ij})}
\\
&
\leq C\varepsilon(N^{-1}\ln N)^{2(k+1)}
	+C(\sq \ln N) e^{-\alpha_2\tau/\varepsilon}
\\
&
\leq C\varepsilon(N^{-1}\ln N)^{2(k+1)}+C\sq N^{-2(k+1)}\ln N \\
&\leq C\varepsilon(N^{-1}\ln N)^{2(k+1)}.
\end{align*}
By similar calculations one obtains
\[
\norm{\eta_{E_{22}}}_{\Omega_{22}}\leq C\sq (N^{-1}\ln N)^{k+1}
\quad\text{and}\quad
\norm{\eta_{E_{22}}}_{\Omega_{21}}\leq C\ssq N^{-(k+1)}\ln^{1/2}N.
\]
These inequalities together yield
$\norm{\eta_u}_{\Omega_{x}}\leq C\ssq(N^{-1}\ln N)^{k+1}$
for the S-mesh. 
The bound on $\norm{\eta_u}_{\Omega_{y}}$ is similar so \eqref{etau:layer:L2} is proved.

One obtains \eqref{etau:2d:1} from the estimates
\begin{align*}
\sum_{j=1}^{N}\norm{(\eta_u)_{i,y}^{-}}_{J_j}^2
&\leq
C\big[N^{-2(k+1)}+N^{-1}(N^{-1}\max|\psi^{\prime}|)^{2k+1}\big]
\ \text{ for } i=1,\dots,N
\intertext{and}
\sum_{j=1}^{N}\norm{(\eta_u)_{x,j}^{-}}_{J_j}^2
&\leq 
C\big[N^{-2(k+1)}+N^{-1}(N^{-1}\max|\psi^{\prime}|)^{2k+1}\big]
\ \text{ for } j=1,\dots,N,
\end{align*}
which appear in the proof of~\cite[Lemma 4.1]{Cheng2020}.
\end{proof}

For each element $K_{ij}\in \Omega_N$ and each $v\in \spc$,
		define the bilinear forms 
		\begin{align*}
		\mathcal{D}^1_{ij}(\eta_u,v)
		&:= \dual{\eta_u}{v_x}_{K_{ij}}
		-\dual{(\eta_u)^{-}_{i,y}}{v^{-}_{i,y}}_{J_{j}}
		+\dual{(\eta_u)^{-}_{i-1,y}}{v^{+}_{i-1,y}}_{J_{j}},
		\\
		\mathcal{D}^2_{ij}(\eta_u,v)
		&:= \dual{\eta_u}{v_y}_{K_{ij}}
		-\dual{(\eta_u)^{-}_{x,j}}{v^{-}_{x,j}}_{I_{i}}
		+\dual{(\eta_u)^{-}_{x,j-1}}{v^{+}_{x,j-1}}_{I_{i}},
		\end{align*}
		with $(\eta_u)^{-}_{0,y}=(\eta_u)^{-}_{x,0}=0$.
The next lemma presents a superapproximation result 
for these bilinear operators.

	\begin{lemma}
		\label{superapproximation:element}
There exists a constant $C$ such that for 
any function $z\in W^{k+2,\infty}(\Omega)$, all $v\in \spc$ and $1\le i,j\le N$, one has
\begin{subequations}
		\begin{align}
		\label{sup:L2}
		|\mathcal{D}^{1}_{ij}(\eta_z,v)|
		&\leq
		Ch_i^{-1}\left[h_{i}^{k+2}\norm{\partial_x^{k+2}z}_{K_{ij}}
		+h_{j}^{k+2}\norm{\partial_y^{k+2}z}_{K_{ij}}
		\right]\norm{v}_{K_{ij}}
		\\
		\label{sup:L0}
		&\leq
		C\sqrt{\frac{h_j}{h_i}}
		\left[h_{i}^{k+2}\norm{\partial_x^{k+2}z}_{L^\infty(K_{ij})}
		+h_{j}^{k+2}\norm{\partial_y^{k+2}z}_{L^\infty(K_{ij})}
		\right]\norm{v}_{K_{ij}},
		\\
		\label{stab:L0:bilinear}
		|\mathcal{D}^{1}_{ij}(\eta_z,v)|
		&\leq C\sqrt{\frac{h_j}{h_i}}
		\norm{z}_{L^\infty(K_{ij})}\norm{v}_{K_{ij}};
		\intertext{furthermore, for the solution $u$ satisfying the bounds of Proposition~\ref{proposition:reg:2d} with $m=k$,}
		\label{sup:L2:bilinear}
		\sum_{K_{ij}\in \Omega_x}
		&\Bigg(\frac{|\mathcal{D}^{1}_{ij}(\eta_u,v)|}{\norm{v}_{K_{ij}}}
		\Bigg)^2
		\leq C\varepsilon^{-1}(N^{-1}\max|\psi^{\prime}|)^{2(k+1)}.
		\end{align}
\end{subequations}
Analogous bounds hold true for $\mathcal{D}^{2}_{ij}(\eta_u,v)$.
	\end{lemma}
\begin{proof}
The estimate \eqref{sup:L2} is proved in \cite[Lemma 4.8]{Zhu:2dMC} for the S-mesh, but the argument remains valid for the BS- and B-meshes. See the proof of~\cite[Theorem 4.1]{Cheng2020}  for \eqref{sup:L0}--\eqref{sup:L2:bilinear}.
\end{proof}

\subsection{The approximation error $\bm \xi$}
\label{subsec:xi}

In this subsection we shall prove two lemmas bounding~$\bm \xi\in \spc^3$, which is the difference between the computed solution and the local Gauss-Radau projection of the exact solution.

The first result is motivated by \cite[Lemma 3.1]{Wang2015}
in which a relationship between the gradient and
the element interface jump of the numerical solution with the numerical solution of the gradient was derived from the element variational equations. 
Using this idea, we obtain a new bound for the derivative and jump of $\xi_u$ from the error equation on the local element 
and the bounds of the projection error $\bm\eta$ 
that were established in Section \ref{subsec:eta}.

	\begin{lemma}\label{relation-to-auxiliary}
There exists a constant $C$ such that 
		\begin{align*}
		\norm{(\xi_u)_x}_{\Omega_{x}}
		+\Bigg(\sum_{j=1}^{N}\sum_{i=N/2+1}^{N}
		h_i^{-1}\norm{\jump{\xi_u}_{i-1,y}}_{J_j}^2\Bigg)^{\frac12}
		&\leq
		C\varepsilon^{-\frac12}
		\left[\enorm{\bm\xi}+(N^{-1}\max|\psi^{\prime}|)^{k+1}\right],
		\\
		\norm{(\xi_u)_y}_{\Omega_{y}}
		+\Bigg(\sum_{i=1}^{N}\sum_{j=N/2+1}^{N}
		h_j^{-1}\norm{\jump{\xi_u}_{x,j-1}}_{I_i}^2\Bigg)^{\frac12}
		&\leq 
		C\varepsilon^{-\frac12}
		\left[\enorm{\bm\xi}+(N^{-1}\max|\psi^{\prime}|)^{k+1}\right].  \notag
		\end{align*}
	\end{lemma}

\begin{proof}
Since these two inequalities are derived in a similar way,
we shall prove only the first one.
Taking $\vphu=\vphq = 0$ in~\eqref{compact:form:2d},
one gets the element variational equation
(see \cite[(2.4b)]{Cheng2020})
\begin{align}\label{variational:equation:cell}
\varepsilon^{-1}\dual{\pN}{\vphp}_{K_{ij}}+
\dual{ \uN}{\vphp_x}_{K_{ij}}
-\dual{\widehat{\uN}_{i,y}}{\vphp^{-}_{i,y}}_{J_j}
+\dual{\widehat{\uN}_{i-1,y}}{\vphp^{+}_{i-1,y}}_{J_j}
=0,
\end{align}
for any $\vphp\in \mathcal{Q}^{k}(K_{ij})$ and $i,j=1,2,\dots,N$,
where $K_{ij}=I_i\times J_j=(x_{i-1},x_i)\times (y_{j-1},y_j)$ 
and the numerical flux $\widehat{\uN}$ is given by 
\[
\widehat{\uN}_{i,y}
=\begin{cases}
\uN^{-}_{i,y} &\text{for } i=1,2,\dots,N-1,\\
0             &\text{for }  i=0,N.\\
\end{cases}
\]
Since the exact solutions $u$ and $p=\varepsilon u_x$ also satisfy
the weak formulation \eqref{variational:equation:cell},
we get the Galerkin orthogonality property
\begin{align*}
\varepsilon^{-1}\dual{e_p}{\vphp}_{K_{ij}}+
\dual{e_u}{\vphp_x}_{K_{ij}}
-\dual{u-\hat U_{i,y}}{\vphp^{-}_{i,y}}_{J_j}
+\dual{u-\hat U_{i-1,y}}{\vphp^{+}_{i-1,y}}_{J_j}
=0,
\end{align*}
for any $\vphp\in \mathcal{Q}^{k}(K_{ij})$ and $i,j=1,2,\dots,N$.
Using the error decomposition \eqref{error:decomposition} 
and an integration by parts, 
for any $\vphp\in \mathcal{Q}^{k}(K_{ij})$ and $i,j=1,2,\dots,N$ we have 
\begin{align}\label{error:equation:cell}
\varepsilon^{-1}\dual{\xi_p}{\vphp}_{K_{ij}}
-\dual{(\xi_u)_x}{\vphp}_{K_{ij}}
&-\dual{\jump{\xi_u}_{i-1,y}}{\vphp^{+}_{i-1,y}}_{J_j}
\nonumber\\
&
=\varepsilon^{-1}\dual{\eta_p}{\vphp}_{K_{ij}}
+\mathcal{D}^{1}_{ij}(\eta_u,\vphp),
\end{align}
where $(\eta_u)^{-}_{0,y}=0$ and $\jump{\xi_u}_{0,y}=(\xi_u)^+_{0,y}$.

Take $\vphp|_{K_{ij}}=\frac{x-x_{i-1}}{h_i}(\xi_u)_x\in \mathcal{Q}^{k}(K_{ij})$ 
in~\eqref{error:equation:cell}. 
Then $\vphp^{+}_{i-1,y}=0$ and
\begin{align*}
\norm{\Big(\frac{x-x_{i-1}}{h_i}\Big)^{1/2}(\xi_u)_x}_{K_{ij}}^2
&=
\dual{\varepsilon^{-1}\xi_p-\varepsilon^{-1}\eta_p}{\vphp}_{K_{ij}}
-\mathcal{D}^{1}_{ij}(\eta_u,\vphp),
\end{align*}
which implies
\begin{align*}
\norm{\Big(\frac{x-x_{i-1}}{h_i}\Big)^{1/2}(\xi_u)_x}_{K_{ij}}
\leq
\varepsilon^{-1}(\norm{\eta_p}_{K_{ij}}+\norm{\xi_p}_{K_{ij}})
+\frac{|\mathcal{D}^{1}_{ij}(\eta_u,\vphp)|}{\norm{\vphp}_{K_{ij}}}
\end{align*}
via a Cauchy-Schwarz inequality.
A scaling argument using norm equivalence on a reference element then gives
\begin{align*}
\norm{(\xi_u)_x}_{K_{ij}}
&\leq 
C\norm{\Big(\frac{x-x_{i-1}}{h_i}\Big)^{1/2}(\xi_u)_x}_{K_{ij}}
\leq 
C\varepsilon^{-1}(\norm{\eta_p}_{K_{ij}}+\norm{\xi_p}_{K_{ij}})
+C\frac{|\mathcal{D}^{1}_{ij}(\eta_u,\vphp)|}{\norm{\vphp}_{K_{ij}}}
\end{align*}
for some constant $C$.
Using the definition of~$\enorm{\bm\xi}$, inequality \eqref{etap:2d:1} of Lemma~\ref{lemma:GR}, and \eqref{sup:L2:bilinear} of Lemma~\ref{superapproximation:element}, 
one obtains 
\begin{align}\label{xi:element}
\norm{(\xi_u)_x}^2_{\Omega_{x}}
=\sum_{j=1}^{N}\sum_{i=N/2+1}^{N}
\norm{(\xi_u)_x}_{K_{ij}}^2
\leq C\varepsilon^{-1}
\left[\enorm{\bm\xi}^2+(N^{-1}\max|\psi^{\prime}|)^{2(k+1)}\right].
\end{align}

We shall now make a different choice of $\vphp|_{K_{ij}}$ in~\eqref{error:equation:cell}. Take $\vphp|_{K_{ij}}=\xi_u(x,y)-\xi_u(x_{i-1}^{-},y)$. Then $\vphp^{+}_{i-1,y}=\jump{\xi_u}_{i-1,y}$.
By a Cauchy-Schwarz inequality we have
\begin{align*}
\left(\int_{x_{i-1}}^{x} (\xi_u)_x(s,y)\,\mathrm{d}s
+\jump{\xi_u}_{i-1,y}\right)^2
\leq 2h_i\int_{x_{i-1}}^{x_i} (\xi_u)^2_x(s,y)\,\mathrm{d}s
+2\jump{\xi_u}^2_{i-1,y},
\end{align*}
for any $x\in I_i$, which leads to
\begin{align}\label{ineq:1}
\norm{\vphp}^2_{K_{ij}}\leq Ch_i^2\norm{(\xi_u)_x}^2_{K_{ij}}
+2h_i\norm{\jump{\xi_u}_{i-1,y}}^2_{J_j}.
\end{align}
Using \eqref{error:equation:cell} and Young's inequality, one has
\begin{align*}
&\norm{\jump{\xi_u}_{i-1,y}}^2_{J_j}
=
\dual{\varepsilon^{-1}\xi_p-\varepsilon^{-1}\eta_p-(\xi_u)_x}{\vphp}_{K_{ij}}
-\mathcal{D}^{1}_{ij}(\eta_u,\vphp)
\\
&\leq
\frac{1}{4}h_i^{-1}\norm{\vphp}^2_{K_{ij}}
+Ch_i\left[
\varepsilon^{-2}(\norm{\eta_p}_{K_{ij}}^2+\norm{\xi_p}_{K_{ij}}^2)
+\norm{(\xi_u)_x}_{K_{ij}}^2
+\left(
\frac{|\mathcal{D}^{1}_{ij}(\eta_u,\vphp)|}{\norm{\vphp}_{K_{ij}}}
\right)^2
\right].
\end{align*}
Combining this inequality with \eqref{ineq:1} gives
\begin{align*}
h_i^{-1}\norm{\jump{\xi_u}_{i-1,y}}^2_{J_j}
\leq &\;
C\left[
\varepsilon^{-2}\left(\norm{\eta_p}_{K_{ij}}^2+\norm{\xi_p}_{K_{ij}}^2\right)
+\norm{(\xi_u)_x}_{K_{ij}}^2
+\left(\frac{|\mathcal{D}^{1}_{ij}(\eta_u,\vphp)|}{\norm{\vphp}_{K_{ij}}}\right)^2
\right].
\end{align*}
Using the definition of~$\enorm{\bm\xi}$, inequality \eqref{etap:2d:1} of Lemma~\ref{lemma:GR},  \eqref{sup:L2:bilinear} of Lemma~\ref{superapproximation:element}
and \eqref{xi:element}, we get
\begin{align}\label{xi:bry}
\sum_{j=1}^{N}\sum_{i=N/2+1}^{N}
h_i^{-1}\norm{\jump{\xi_u}_{i-1,y}}_{J_j}^2
\leq
C\varepsilon^{-1}\left[\enorm{\bm\xi}^2
+(N^{-1}\max|\psi^{\prime}|)^{2(k+1)}\right].
\end{align}

The desired inequality now follows from \eqref{xi:element} and \eqref{xi:bry}.
\end{proof}

Our second bound on $\bm\xi$  deals with the element boundary error.

\begin{lemma}\label{lemma:edge:estimate} 
There exists a constant $C$ such that
\begin{align*}
\sum_{j=1}^N\norm{(\xi_u)^+_{i,y}}^2_{J_j}
&\leq C\tau\sq^{-1} \left[
\enorm{\bm\xi}^2+(N^{-1}\max|\psi^{\prime}|)^{2(k+1)}\right]
\ \text{ for } i=N/2,\dots,N-1,
		\\
\sum_{i=1}^N\norm{(\xi_u)^+_{x,j}}^2_{I_i}
&\leq C\tau\sq^{-1} \left[
\enorm{\bm\xi}^2+(N^{-1}\max|\psi^{\prime}|)^{2(k+1)}\right]
\ \text{ for } j=N/2,\dots,N-1.
\end{align*}
\end{lemma}

\begin{proof}
We prove the first inequality; the second is similar.
For $i=N/2,\dots,N-1$ and $y\in J_j$ ($j=1,2,\dots,N$), one can write 
\[
(\xi_u)^+_{i,y}
= -\sum_{\ell=i+1}^N\int_{x_{\ell-1}}^{x_\ell}(\xi_u)_x(s,y)\mathrm{d}s
- \sum_{\ell=i+1}^{N-1}\jump{\xi_u}_{\ell,y}
+(\xi_u)^{-}_{N,y}.
\]
Then Cauchy-Schwarz inequalities give
	\begin{align*}
	|(\xi_u)^+_{i,y}|^2
	&\leq 
	3\Bigg[\sum_{\ell=i+1}^N\int_{I_{\ell}}|(\xi_u)_x|\mathrm{d}x\Bigg]^2
	+3\Bigg[\sum_{\ell=i+1}^{N-1}\jump{\xi_u}_{\ell,y}\Bigg]^2
	+3\jump{\xi_u}^2_{N,y}
	\\
	&\leq  
	3\Bigg[\sum_{\ell=i+1}^N
	h_\ell^{\frac12}\cdot\Big(\int_{I_{\ell}}|(\xi_u)_x|^2\mathrm{d}x\Big)^{\frac12} 
	\Bigg]^2
	+3\Bigg[\sum_{\ell=i+1}^{N-1} \Big(h^{\frac12}_{\ell+1} \cdot h^{-\frac12}_{\ell+1}\jump{\xi_u}_{\ell,y}\Big)\Bigg]^2
	+3\jump{\xi_u}^2_{N,y}
	\\
	&\leq  
	3\Bigg[\sum_{\ell=i+1}^N h_\ell\Bigg]
	\cdot\Bigg[\sum_{\ell=i+1}^N
	\int_{I_{\ell}}|(\xi_u)_x|^2\mathrm{d}x\Bigg]
	+3\Bigg[\sum_{\ell=i+1}^{N-1} h_{\ell+1}\Bigg]
	\cdot\Bigg[\sum_{\ell=i+1}^{N-1}  
	h^{-1}_{\ell+1}\jump{\xi_u}^2_{\ell,y}\Bigg]
	+3\jump{\xi_u}^2_{N,y}
	\\
	&\leq  
	3\tau\Bigg[
	\sum_{\ell=i+1}^N \int_{I_\ell}|(\xi_u)_x|^2\mathrm{d}x
	+\sum_{\ell=i+1}^{N-1} h^{-1}_{\ell+1}\jump{\xi_u}^2_{\ell,y}
	\Bigg]
	+3\jump{\xi_u}^2_{N,y},
	\end{align*}
	for $i=N/2,\dots,N-1$.
Using Lemma \ref{relation-to-auxiliary}, one gets
\begin{align*}
\sum_{j=1}^N\norm{(\xi_u)^+_{i,y}}^2_{J_j}
&\leq C\tau
\Bigg[\sum_{j=1}^N\sum_{\ell=i+1}^N \norm{(\xi_u)_x}^2_{K_{\ell j}}
+\sum_{j=1}^N\sum_{\ell=i+1}^{N-1} h^{-1}_{\ell+1}
\norm{\jump{\xi_u}_{\ell,y}}_{J_j}^2
\Bigg]
+3\sum_{j=1}^N\norm{\jump{\xi_u}_{N,y}}_{J_j}^2
\\
&\leq 
C\tau\sq^{-1}\Big(\enorm{\bm\xi}^2
+(N^{-1}\max|\psi^{\prime}|)^{2(k+1)}\Big)
+\frac{3}{\min\limits_{y\in[0,1]}\big[\frac12(\af)_{N,y}+\lambda_{1}\big]}\enorm{\bm\xi}^2,
\end{align*}
for $i=N/2,\dots,N-1$.
The desired result now follows since $\lambda_1\ge 0$, $\af(x,y)\geq \alpha_1>0$ in $\overline{\Omega}$, and $\tau\geq C\varepsilon$.
\end{proof}

\section{The supercloseness result}\label{sec:proof}

We now state and prove  the main result of the paper.

\begin{theorem}\label{thm:superconvergent}
Recall Assumptions~\ref{ass:1} and~\ref{ass:2}.
	Let $\bm w=(u,p,q)=(u,\varepsilon u_x,\varepsilon u_y)$
	be the solution of problem \eqref{cd:spp:2d}.
	Let $\wN=(\uN,\pN,\qN)\in \spc^3$ be the numerical solution of the LDG method \eqref{compact:form:2d}.
	Then for some constant $C>0$ one has the superclose property
    \begin{align}\label{super:energy:error}
    \enorm{\bm\Pi \bm w-\wN}\leq CM^{\star}N^{-(k+1)},
    \end{align}
	where $\bm\Pi \bm w \in \spc^3$ is the local Gauss-Radau projection of $\bm w $ and
	\begin{align}\label{bounding:Q}
	M^{\star}
	:=\sqrt{\frac{\mu}{\sq}}(\max|\psi^{\prime}|)^{k+1}
	=\begin{cases}
	(\ln N)^{k+1}                                  
	&\text{for the S-mesh},\\
	(\ln N)^{1/2}                                                 
	&\text{for the  BS-mesh},\\
	(\ln (1/\varepsilon))^{1/2}
	&\text{for the B-mesh}.
	\end{cases}
	\end{align}
	Furthermore, one has the $L^2$ error estimate
	\begin{align}\label{optimal:L2:error}
	\lnorm{\bm w-\wN} \leq CM^{\star}N^{-(k+1)}.
	\end{align}
\end{theorem}
\begin{proof}
Recall that  $\bm \xi = \wN-\bm\Pi \bm w$ and $\bm\eta =  \bm w-\bm\Pi \bm w$.
Using the definition of $\enorm{\cdot}$
and Galerkin orthogonality,
one obtains
\begin{align}\label{error:equation}
\enorm{\bm\xi}^2 =
B(\bm \xi;\bm\xi)=B(\bm \eta;\bm\xi)
=\sum_{i=1}^3\mathcal{T}_i(\bm \eta;\bm\xi)+\mathcal{T}_4(\eta_u;\xi_u),
\end{align}
where the $\mathcal{T}_i\,(i=1,2,3,4)$ are defined in \eqref{B:def:2d}.
From \cite[Theorem 4.1]{Cheng2020}, we have
\begin{align}\label{Ti}
|\mathcal{T}_i(\bm \eta;\bm\xi)|
&\leq C(N^{-1}\max|\psi^{\prime}|)^{k+1} \enorm{\bm\xi}
\ \text{ for } i=1,2,3,
\\
\label{bound:T4:old}
|\mathcal{T}_4(\eta_u;\xi_u)|
&\leq CQ^{\star}N^{-(k+1/2)} \enormsharp{\bm\xi},
\end{align}
where $\enormsharp{\cdot}$ (see~\cite[eq.~(4.13)]{Cheng2020}) is 	a stronger norm than the energy norm $\enorm{\cdot}$ 
and $Q^{\star}$ may depend weakly on $N$ and $\varepsilon$ (see~\cite[eq.~(4.9)]{Cheng2020}).
Clearly the estimate \eqref{bound:T4:old} of the convection term $\mathcal{T}_4(\eta_u;\xi_u)$ needs to be improved to yield our supercloseness result, and we do this now.
The crux of the argument is to derive improved estimates
of the derivative and the jump of $\xi_u$;
here Lemma \ref{relation-to-auxiliary} plays an important role,
as we shall see when proving~\eqref{T43} below.

Consider the $x$-direction component of $\mathcal{T}_4(\eta_u;\xi_u)$,
which is defined by
\begin{align*}
\mathcal{T}^x_4(\eta_u;\xi_u)
&:= -\dual{\af\eta_u}{(\xi_{u})_{x}}
-\sum_{j=1}^{N}\Big[\sum_{i=1}^{N}
\dual{(\af)_{i,y}(\eta_u)^{-}_{i,y}}{\jump{\xi_u}_{i,y}}_{J_j}
-\dual{\lambda_{1}(\eta_u)^{-}_{N,y}}{(\xi_u)^{-}_{N,y}}_{J_j}
\Big]
\\
&:= \sum_{i=1}^5 \mathcal{T}_{4i}(\eta_u;\xi_u),
\end{align*}
where we set $(\af)_{ij}:=\af(x_i,y_j)$ and define
\begin{align*}
\mathcal{T}_{41}(\eta_u;\xi_u) &:=
-\sum_{j=1}^N\sum_{i=1}^{N/2} (\af)_{ij}\mathcal{D}^1_{ij}(\eta_u,\xi_u),
\\
\mathcal{T}_{42}(\eta_u;\xi_u) &:=
-\sum_{j=1}^N\sum_{i=1}^{N/2}
\Big[\dual{(\af-(\af)_{ij})\eta_u}{(\xi_{u})_{x}}_{K_{ij}}
-\dual{((\af)_{i,y}-(\af)_{ij})(\eta_u)^{-}_{i,y}}{(\xi_u)^{-}_{i,y}}_{J_{j}}
\\
&\hspace{3cm}
+\dual{((\af)_{i-1,y}-(\af)_{ij})(\eta_u)^{-}_{i-1,y}}{(\xi_u)^{+}_{i-1,y}}_{J_{j}}\Big],
\\
\mathcal{T}_{43}(\eta_u;\xi_u) &:=
-\sum_{j=1}^N\sum_{i=N/2+1}^N
\dual{\af\eta_u}{(\xi_{u})_{x}}_{K_{ij}}
-\sum_{j=1}^N\sum_{i=N/2+1}^N
\dual{(\af)_{i,y}(\eta_u)^{-}_{i,y}}{\jump{\xi_u}_{i,y}}_{J_j},
\\
\mathcal{T}_{44}(\eta_u;\xi_u) &:=
-\sum_{j=1}^N (\af)_{N/2,j}
\dual{(\eta_u^{-})_{N/2,y}}{(\xi_u^{+})_{N/2,y}}_{J_j},
\\
\mathcal{T}_{45}(\eta_u;\xi_u) &:=
\sum_{j=1}^N\lambda_1\dual{(\eta_u^{-})_{N,y}}{(\xi_u^{-})_{N,y}}_{J_j}.
\end{align*}
Each of these terms will be estimated separately.
For $\mathcal{T}_{41}(\eta_u;\xi_u)$ we use 
the technique that 
bounded the term $\mathcal{T}_2(\cdot;\cdot)$ 
in \cite[Theorem 4.1]{Cheng2020},
but keeping in mind that the mesh size in the $x$ direction 
is now~$O(N^{-1})$.
For example, since $\sigma\geq k+2$, from  \eqref{sup:L0} and \eqref{stab:L0:bilinear}
of Lemma~\ref{superapproximation:element} and Cauchy-Schwarz inequalities we get
\begin{align}
\left|\mathcal{T}_{41}(\eta_S;\xi_u)\right|
&\leq C\sum_{j=1}^N\sum_{i=1}^{N/2}
\sqrt{\frac{h_{j}}{h_{i}}}\,
\Big[h_{i}^{k+2}
\norm{\partial_x^{k+2}S}_{L^\infty(K_{ij})}
+h_{j}^{k+2}\norm{\partial_y^{k+2}S}_{L^\infty(K_{ij})}
\Big]\norm{\xi_u}_{K_{ij}}
\nonumber\\
&\leq C\sum_{j=1}^N\sum_{i=1}^{N/2} N^{-(k+2)}\norm{\xi_u}_{K_{ij}}
\leq CN^{-(k+1)}\enorm{\bm\xi},
\nonumber\\
\left|\mathcal{T}_{41}(\eta_{E_{21}};\xi_u)\right|
&\leq C\sum_{j=1}^N\sum_{i=1}^{N/2}
\sqrt{\frac{h_{j}}{h_{i}}}
\norm{E_{21}}_{L^\infty(K_{ij})}
\norm{\xi_u}_{K_{ij}}
\leq C \left[\psi\left(\frac12\right)\right]^{\sigma} \sum_{j=1}^N\sum_{i=1}^{N/2} \norm{\xi_u}_{K_{ij}}
\nonumber\\
\label{T41:E21}
&
\leq CN \left[\psi\left(\frac12\right)\right]^{\sigma}\norm{\xi_u}
\leq CN^{-(k+1)}\enorm{\bm\xi}.
\end{align}
To bound $\mathcal{T}_{41}(\eta_{E_{12}};\xi_u)$,
we split it into two sums and estimate each separately.
Similarly to above,
	\begin{align*}
	&\left|\sum_{j=1}^{N/2}\sum_{i=1}^{N/2} (\af)_{ij}\mathcal{D}^1_{ij}(\eta_{E_{12}},\xi_u)\right|
	\leq C  \left[\psi\left(\frac12\right)\right]^{\sigma} \sum_{j=1}^{N/2}\sum_{i=1}^{N/2}
	\norm{\xi_u}_{K_{ij}}
	\leq CN^{-(k+1)}\enorm{\bm\xi}.
	\end{align*}
Using $\sigma\geq k+2$, \eqref{sup:L0}--\eqref{stab:L0:bilinear}, 
and $h_j/\varepsilon \geq CN^{-1}\min\varphi^{\prime}\geq CN^{-1}\geq h_i$
for $1\leq i\leq N/2,\ N/2+1\leq j\leq N$ from \eqref{mesh:size:layer},
we get
	\label{fnt:taupsi}
	\begin{align*}
	&\left|\sum_{j=N/2+1}^{N}\sum_{i=1}^{N/2}  (\af)_{ij}\mathcal{D}^1_{ij}(\eta_{E_{12}},\xi_u)\right|
	\\
	&
	\leq C\sum_{K_{ij}\in \Omega_{12}}
	\sqrt{\frac{h_{j}}{h_{i}}}
	\min\Big\{
	h_{i}^{k+2}\norm{\partial_x^{k+2}E_{12}}_{L^\infty(K_{ij})}
	+h_{j}^{k+2}\norm{\partial_y^{k+2}E_{12}}_{L^\infty(K_{ij})},
	\norm{E_{12}}_{L^\infty(K_{ij})}
	\Big\}\norm{\xi_u}_{K_{ij}}
	\\
	&\leq 
	C\sum_{K_{ij}\in \Omega_{12}}\sqrt{\frac{h_{j}}{h_{i}}}
	\min\left\{1,\left(\frac{h_{j}}{\varepsilon}\right)^{k+2}\right\}
	e^{-\alpha_2(1-y_j)/\varepsilon}\norm{\xi_u}_{K_{ij}}
	\\
	&\leq
	CN^{1/2}\left(\sum_{i=1}^{N/2}\sum_{j=N/2+1}^{N}h_j\right)^{1/2}
	\left(\max_{N/2+1\leq j\leq N}\Theta_j\right)^{k+2}\norm{\xi_u}
	\\
	&\leq
	CN\tau^{1/2}
	(N^{-1}\max|\psi^{\prime}|)^{k+2}\norm{\xi_u}
	\leq
	C(N^{-1}\max|\psi^{\prime}|)^{k+1}\enorm{\bm\xi},
	\end{align*}
	where $\Theta_j := e^{-\alpha_2(1-y_j)/[(k+2)\varepsilon]} \min\left\{1,h_{j}/\varepsilon\right\}
	\leq CN^{-1}\max|\psi^{\prime}|$ by~\cite[(3.6a)]{Cheng2020}.
	In the last inequality, we used  
	$\tau^{1/2}\max|\psi^{\prime}| \leq C$.
	This is trivial for the BS-mesh and B-mesh 
	(see Table~\ref{table:functions}).
    For the S-mesh, this inequality follows  from Assumption~\ref{ass:1} (ii) and
    $\tau^{1/2}\max|\psi^{\prime}|= C\ssq\ln^{3/2} N
    \leq CN^{-1/2}\ln^{3/2} N\leq C$ for $N\geq 4$.

Similarly to \eqref{T41:E21}, one has $\left|\mathcal{T}_{41}(\eta_{E_{22}};\xi_u)\right|
\leq CN^{-(k+1)}\enorm{\bm\xi}$. 
Combining these bounds, we get
\begin{equation}\label{T41}
\left|\mathcal{T}_{41}(\eta_{u};\xi_u)\right|
\leq C(N^{-1}\max|\psi^{\prime}|)^{k+1}\enorm{\bm\xi}.
\end{equation}

Next, a  Cauchy-Schwarz inequality, 
$\af-(\af)_{ij}=O(N^{-1})$ for $(x,y)\in K_{ij}$,
an inverse inequality 
and Lemma~\ref{lemma:GR} yield
\begin{align}\label{T42}
\left|\mathcal{T}_{42}(\eta_u;\xi_u)\right|
&\leq
C\sum_{j=1}^N\sum_{i=1}^{N/2} N^{-1}\Big(\norm{\eta_u}_{K_{ij}}\norm{(\xi_u)_x}_{K_{ij}}
\nonumber\\
&\hspace{2cm}
+\norm{(\eta_u)^{-}_{i,y}}_{J_{j}}
\norm{(\xi_u)^{-}_{i,y}}_{J_{j}}
+\norm{(\eta_u)^{-}_{i-1,y}}_{J_{j}}
\norm{(\xi_u)^{+}_{i-1,y}}_{J_{j}}
\Big)
\nonumber\\
&\leq C\Bigg(\sum_{j=1}^N\sum_{i=1}^{N/2}
(\norm{\eta_u}^2_{K_{ij}}
+N^{-1}\norm{(\eta_u)^{-}_{i,y}}^2_{J_{j}}
)\Bigg)^{1/2}\norm{\xi_u}
\nonumber\\
&\leq
C(N^{-1}\max|\psi^{\prime}|)^{k+1} \enorm{\bm\xi}.
\end{align}

The term $\mathcal{T}_{43}(\eta_{u};\xi_u)$ 
handles the convective error in the layer region
which is our main concern.
We shall use  Lemma \ref{relation-to-auxiliary}
to absorb the derivative and jump of the projection error
into the energy norm and improve the final convergence rate.
Invoking Lemma \ref{relation-to-auxiliary} 
and Lemma \ref{lemma:GR}, we get
\begin{align}
&\left|\mathcal{T}_{43}(\eta_{u};\xi_u)\right| \notag\\
&\leq C\norm{\eta_u}_{\Omega_x}\norm{(\xi_u)_x}_{\Omega_x}
\nonumber\\
&\qquad
+C
\left(\sum_{j=1}^{N}\sum_{i=N/2+1}^{N-1}
h_{i+1}\norm{(\eta_u)_{i,y}^{-}}_{J_j}^2
\right)^{\frac12}
\left(\sum_{j=1}^{N}\sum_{i=N/2+1}^{N-1}
h_{i+1}^{-1}\norm{\jump{\xi_u}_{i,y}}_{J_j}^2\right)^{\frac12}
\nonumber\\
&\qquad
+C\left(\sum_{j=1}^{N}
\norm{(\eta_u)_{N,y}^{-}}_{J_j}^2
\right)^{\frac12}
\left(\sum_{j=1}^{N}
\norm{\jump{\xi_u}_{N,y}}_{J_j}^2\right)^{\frac12}
\nonumber\\
&\leq
C\varepsilon^{-1/2}\left[\norm{\eta_u}_{\Omega_x}
+\left(\sum_{j=1}^{N}\sum_{i=N/2+1}^{N-1}
h_{i+1}\norm{(\eta_u)_{i,y}^{-}}_{J_j}^2\right)^{1/2}
\right]
\left(\enorm{\bm\xi}+(N^{-1}\max|\psi^{\prime}|)^{k+1}\right)
\nonumber\\
&\qquad
+C\left(\sum_{j=1}^{N}
\norm{(\eta_u)_{N,y}^{-}}_{J_j}^2
\right)^{\frac12}\enorm{\bm\xi}
\nonumber\\
&
\leq C \sqrt{\frac{\mu}{\sq}}
(N^{-1}\max|\psi^{\prime}|)^{k+1}
\left(\enorm{\bm\xi}+(N^{-1}\max|\psi^{\prime}|)^{k+1}\right),  \label{T43}
\end{align}
where in the last inequality we used \eqref{etau:layer:L2}, and also \eqref{etau:2d:1} to bound 
\begin{align*}
\varepsilon^{-1/2}&\left(\sum_{j=1}^{N}\sum_{i=N/2+1}^{N-1}
h_{i+1}\norm{(\eta_u)_{i,y}^{-}}_{J_j}^2\right)^{1/2}
\\
&\leq 
\varepsilon^{-1/2}\left(\max_{N/2+1\leq i\leq N-1}\sum_{j=1}^{N}
\norm{(\eta_u)_{i,y}^{-}}_{J_j}^2\right)^{1/2}
\left(\sum_{i=N/2+1}^{N-1} h_{i+1}\right)^{1/2}
\\
&\leq  C\sqrt{\frac{\tau}{\sq}}
N^{-(k+1)}(\max|\psi^{\prime}|)^{k+1/2}
=C\sqrt{\frac{\mu}{\sq}}(N^{-1}\max|\psi^{\prime}|)^{k+1}
\end{align*}
for our three layer-adapted meshes.

Lemma \ref{lemma:edge:estimate} and \eqref{etau:2d:1} 
of Lemma~\ref{lemma:GR} give
\begin{align}
\left|\mathcal{T}_{44}(\eta_{u};\xi_u)\right|
&\leq C\Bigg(\sum_{j=1}^{N}\norm{(\eta_u)_{N/2,y}^{-}}_{J_j}^2
\Bigg)^{1/2}
\Bigg(\sum_{j=1}^{N}\norm{(\xi_u)_{N/2,y}^{+}}_{J_j}^2
\Bigg)^{1/2}
\nonumber\\
&
\leq CN^{-(k+1)}(\max|\psi^{\prime}|)^{k+1/2}
\Big[\tau\sq^{-1}
\Big(\enorm{\bm\xi}^2+(N^{-1}\max|\psi^{\prime}|)^{2(k+1)}
\Big)\Big]^{1/2}
\nonumber\\
&
= C\sqrt{\frac{\mu}{\sq}}(N^{-1}\max|\psi^{\prime}|)^{k+1}
\left[\enorm{\bm\xi}+(N^{-1}\max|\psi^{\prime}|)^{k+1}\right].  
\label{T44}
\end{align}

Since $0\leq \lambda_{1}\leq C$, a Cauchy-Schwarz inequality 
and Lemma~\ref{lemma:GR} yield
\begin{align}
\left|\mathcal{T}_{45}(\eta_u;\xi_u)\right|
&
\le C\Bigg(\sum_{j=1}^N\norm{(\eta_u)^{-}_{N,y}}^2_{J_j}\Bigg)^{1/2}
\Bigg(\sum_{j=1}^N\norm{\jump{\xi_u}_{N,y}}^2_{J_j}\Bigg)^{1/2}  \notag\\
&\leq C(N^{-1}\max|\psi^{\prime}|)^{k+1} \enorm{\bm\xi}.  \label{T45}
\end{align}

Putting together \eqref{T41}--\eqref{T45}, we have shown that 
\[
\left|\mathcal{T}^x_{4}(\eta_u;\xi_u)\right|
\leq C\sqrt{\frac{\mu}{\sq}}(N^{-1}\max|\psi^{\prime}|)^{k+1}
	\left[\enorm{\bm\xi}
	+(N^{-1}\max|\psi^{\prime}|)^{k+1}\right].
\]
We can bound $\mathcal{T}^y_{4}(\eta_u;\xi_u)$ analogously
and thus we obtain 
\begin{equation}\label{T4}
\left|\mathcal{T}_{4}(\eta_u;\xi_u)\right|
\leq C\sqrt{\frac{\mu}{\sq}}(N^{-1}\max|\psi^{\prime}|)^{k+1}
	\left[\enorm{\bm\xi}
	+(N^{-1}\max|\psi^{\prime}|)^{k+1}\right].
\end{equation}

From \eqref{error:equation}, \eqref{Ti}
and \eqref{T4} it follows  that
\begin{equation}\label{super:xi:1}
\enorm{\bm\xi}\leq C\sqrt{\frac{\mu}{\sq}}(N^{-1}\max|\psi^{\prime}|)^{k+1} 
\end{equation}
and we have proved the superclose estimate \eqref{super:energy:error}.

The $L^2$-error estimate \eqref{optimal:L2:error} is now immediate from \eqref{super:xi:1}  and 
Lemma~\ref{lemma:GR}  since $ \bm w - \bm W = \bm \eta-\bm \xi$.
\end{proof}

\begin{remark}
Several energy-norm error estimates are known for DG methods on the S-mesh.
For piecewise polynomials of degree $k=1$,
the energy-norm error $\enorm{u^I-\uN}$ of the NIPG method,
where $u^I$ is the piecewise bilinear interpolant of $u$,
is $O(N^{-1}\ln^{3/2} N)$ --- see \cite[Theorem 2]{Roos2003}.
This convergence rate is slightly improved to $O(N^{-1}\ln N)$
for the LDG method with penalty flux in \cite[Lemma 4.2]{Zhu2013}.
For piecewise polynomials of degree $k>1$ in the LDG method~\eqref{compact:form:2d},
the energy-norm error $\enorm{\bm\Pi\bm w-\wN}$ 
is shown to be $O((N^{-1}\ln N)^{k+1/2})$ in \cite[Theorem 3.1]{Zhu:2dMC}.
Recently, in \cite[(4.22)]{Cheng2020}
a generalised convergence rate estimate $Q^{\star}N^{-(k+1/2)}$ for $\enorm{\bm\Pi\bm w-\wN}$ 
was established on the three layer-adapted meshes considered in this paper.
Consequently, the energy-norm error $\enorm{\bm w-\wN}$ is 
$O(Q^{\star}N^{-(k+1/2)})$,  which is shown to be sharp (up to a logarithmic factor) in numerical experiments.
This shows that the convergence rate $M^{\star}N^{-(k+1)}$ of~\eqref{super:energy:error} is a superclose-type superconvergence result.
\end{remark}

\begin{remark}\label{rem:optimalL2}
The $L^2$ error bound of \eqref{optimal:L2:error} is optimal
for the S-mesh, and it is optimal up to logarithmic factors for the BS-mesh and B-mesh.
\end{remark}

\begin{remark}\label{comments:ep:N}
	If we discard the assumption $\varepsilon\leq N^{-1}$ of Assumption~\ref{ass:1}(ii),
how does this affect our analysis? 
	 Instead of Theorem~\ref{thm:superconvergent}, 
	we obtain for all three layer-adapted meshes
	 the  general estimate 
	 (provided we make the mild assumption $\varepsilon\ln^3 N\leq C$ for the S-mesh)
	\begin{align}\label{a:general:estimate}
	\enorm{\bm\xi}&\leq
	C\sqrt{\frac{\mu}{\sq}}
	\Bigg\{
	\Big(1+(Nh)^{3/2}\Big)h^{k+1}
	+\Big(1+N(Nh)^{1/2}\Big)\left[\psi\left(\frac12\right)\right]^{\sigma}
	\nonumber\\
	&+\Big(1+(Nh)^{1/2}\Big)(N^{-1}\max|\psi^{\prime}|)^{k+1} 
	\Bigg\},
	\end{align}
	where 
	the first term comes  mainly from the global approximation 
	of the smooth component of the solution
	using the maximum mesh size $h:=\max_i h_i$,
	while the other two terms come from the smallness of the layer component
	in the coarse domain and 
	the approximation of the layer component in the refined region.
	The factor $Nh$ arises from the maximum ratio 
	$\max_{i,j}h_j/h_i$ that appears in \eqref{sup:L0}--\eqref{stab:L0:bilinear}.

	Let us discuss this estimate \eqref{a:general:estimate}
	for the three layer-adapted meshes.
	For the S-mesh one has
	$h\leq 2N^{-1}$, $\psi(1/2)=N^{-1}$, and  \eqref{a:general:estimate} yields 
	Theorem \ref{thm:superconvergent} 	provided $\sigma\geq k+2$.
	For the B-mesh one has $h\leq CN^{-1}$, and 
	if $\psi(1/2)=\varepsilon\leq N^{-\theta}$
	for some constant $0<\theta\leq 1$,  so
    Theorem \ref{thm:superconvergent} holds true
    provided $\sigma\geq (k+2)/\theta$.
	Finally, for the BS-mesh, 
	one has $\psi(1/2)=N^{-1}$ and $h=\max\{\varepsilon,N^{-1}\}$,
	which yields the bound
	\begin{align*}
	\enorm{\bm\xi}&\leq
	C(\ln N)^{1/2}(1+\varepsilon N)^{3/2}(\varepsilon+N^{-1})^{k+1},
	\end{align*}
	whose convergence behavior is unclear when $\varepsilon\geq CN^{-1}$. 
\end{remark}

\begin{remark}
One can combine Lemma \ref{lemma:edge:estimate}, \eqref{etau:2d:1} 
and \eqref{super:xi:1} to derive another superconvergence property:
\[
\left(\sum_{j=1}^N\norm{(e_u^{-})_{i,y}}^2_{J_j}\right)^{1/2}
+\left(\sum_{i=1}^N\norm{(e_u^{-})_{x,j}}^2_{I_i}\right)^{1/2}
\leq C \mu\sq^{-1} N^{-(k+1)}(\max|\psi^{\prime}|)^{k+3/2}
\]
for any $i,j \in \{N/2,\dots,N-1\}$.
\end{remark}

\section{Numerical experiments}
\label{sec:experiments}
\setcounter{equation}{0}

In this section, we illustrate the numerical performance
of the LDG method on the layer-adapted meshes of Section~\ref{subsec:layer:adapted:meshes} for a
two-dimensional singularly perturbed convection-diffusion test problem.
In our experiments 
we take $\alpha_1=1,\alpha_2=2$	 and $\sigma = k+2$
with the penalty parameters $\lambda_1=\lambda_{2}=0$.
The discrete linear systems are solved using LU decomposition, i.e., a direct linear solver.
All integrals are evaluated using the 5-point Gauss-Legendre quadrature rule.

We will present results for the three errors
\begin{align*}
\lnorm{\bm{w}-\wN},\quad \enorm{\Pi \bm w-\wN}, \quad \enorm{\bm w-\wN},
\end{align*}
together with their respective convergence rates, which are calculated from
\[
	r_2 := \frac{ \log (E_N/E_{2N}) }{\log 2 }
	\quad \textrm{or} \quad
	r_S := \frac{ \log (E_N/E_{2N}) }{\log \big(2\ln N/\ln(2N)\big) },
\]
where $E_N$ is the observed error when $N$ elements are used in each coordinate direction.
The quantities $r_2$ and $r_S$ measure the convergence rates from error bounds of the form $CN^{-r_2}$ and $C(N^{-1}\ln N)^{r_S}$ respectively.

\begin{example}\label{exa:1}
Consider the convection-diffusion problem
\begin{align*}			
-\varepsilon \Delta u + (2-x)u_x + (3-y^3) u_y + u &= f \ 
\mathrm{in} \  \Omega=(0,1)^2,
\\
u & = 0 \ \mathrm{on}\ \partial \Omega,
\end{align*}
with $f$ chosen such that
\[
u(x,y)
=  \left(1-e^{-(1-x)/\varepsilon}\right)y^3\left(1-e^{-2(1-y)/\varepsilon}\right)\sin x
\]
is the exact solution. This exact solution has precisely the layer behaviour that one expects in typical solutions of~\eqref{cd:spp:2d}.
\end{example}

Fix $k=2$, $N=32$ and $\varepsilon=10^{-2}$.
Figure~\ref{figure:1} displays plots of the numerical solution~$U$ 
and the error $u-U$ on the three types of layer-adapted meshes
computed by the LDG method.
We see that the LDG method yields good numerical approximations.
Furthermore, no oscillations are visible in the solution.
The largest errors appearing on the outflow boundary
decrease quickly in one element.
This demonstrates the ability of the LDG method to capture the boundary layers. 

In Tables~\ref{table:error:S}-\ref{table:error:B} we take $\varepsilon=10^{-8}$ and 
present the values of $\lnorm{\bm{w}-\wN}, \enorm{\Pi \bm w-\wN}$ and $\enorm{\bm w-\wN}$ and their
convergence rates on the S-mesh, BS-mesh and B-mesh. 
These numerical results confirm the supercloseness of the projection $\Pi\bm w$ 
and the optimal $L^2$ error estimate of Theorem~\ref{thm:superconvergent}.
Comparing these errors, one finds that the S-mesh produces the largest error while the BS-mesh and B-mesh
have comparable errors. 
A careful inspection reveals subtle differences between the last two meshes: 
for the values of $N$ that we used, 
the BS-mesh errors are slightly smaller than the B-mesh errors,
but the convergence rate on the B-mesh is better than the BS-mesh.

In Tables \ref{table:robust:S}-\ref{table:robust:B} we take $k=2$ and $N=128$, and test the robustness of the errors with respect to the singular perturbation parameter~$\varepsilon$. On the S-mesh and BS-mesh the errors vary only slightly with~$\varepsilon$.
On the B-mesh, the definition of~$M^{\star}$ in~\eqref{super:energy:error} indicates a theoretical dependence on~$\varepsilon$, 
but we see in Table~\ref{table:robust:B} that this dependence is insignificant.

\begin{figure}[h]
	\begin{minipage}{0.49\linewidth}
		\centerline{\includegraphics[width=2.6in,height=2.4in]{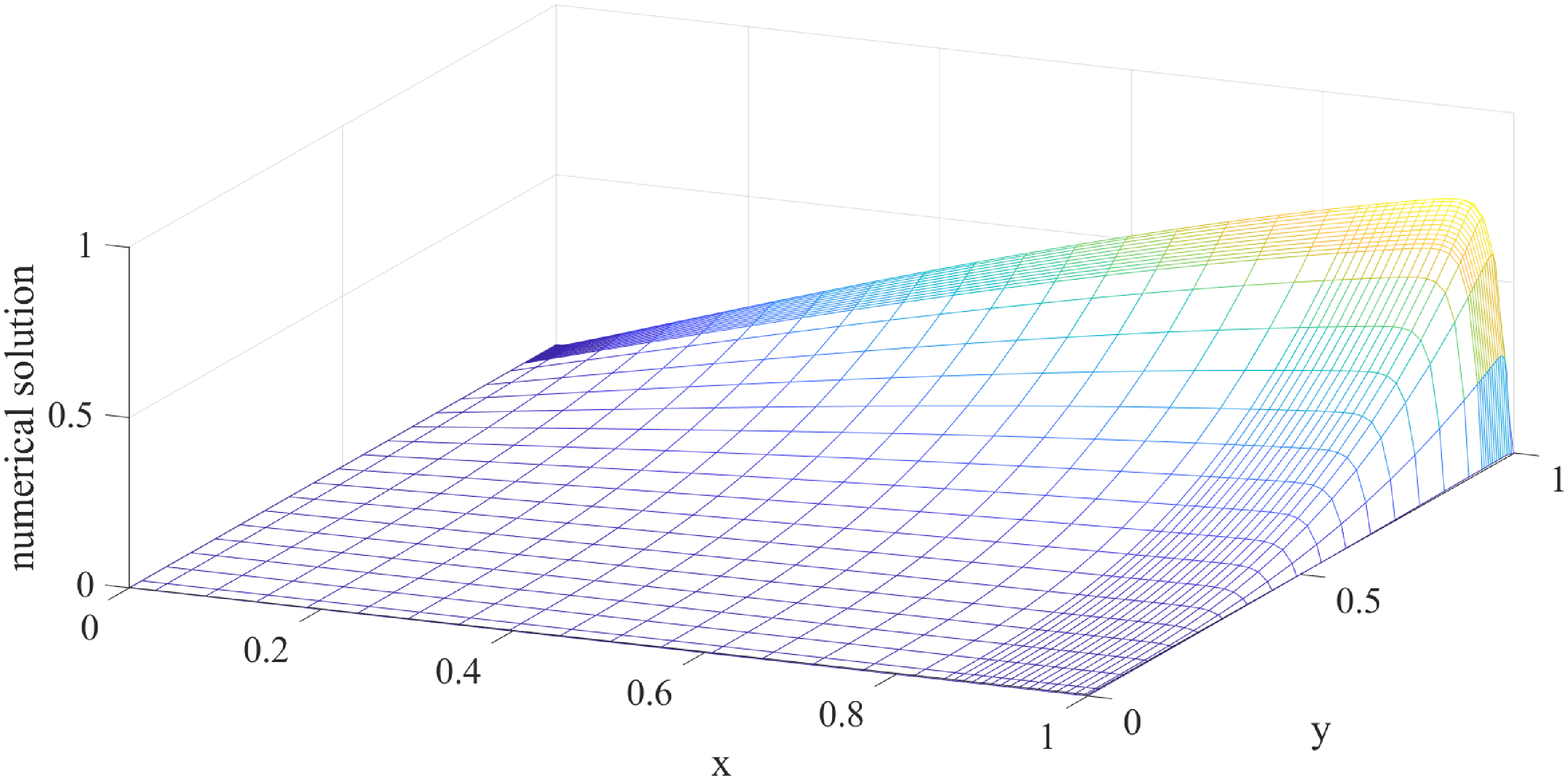}}
		\centerline{\includegraphics[width=2.6in,height=2.4in]{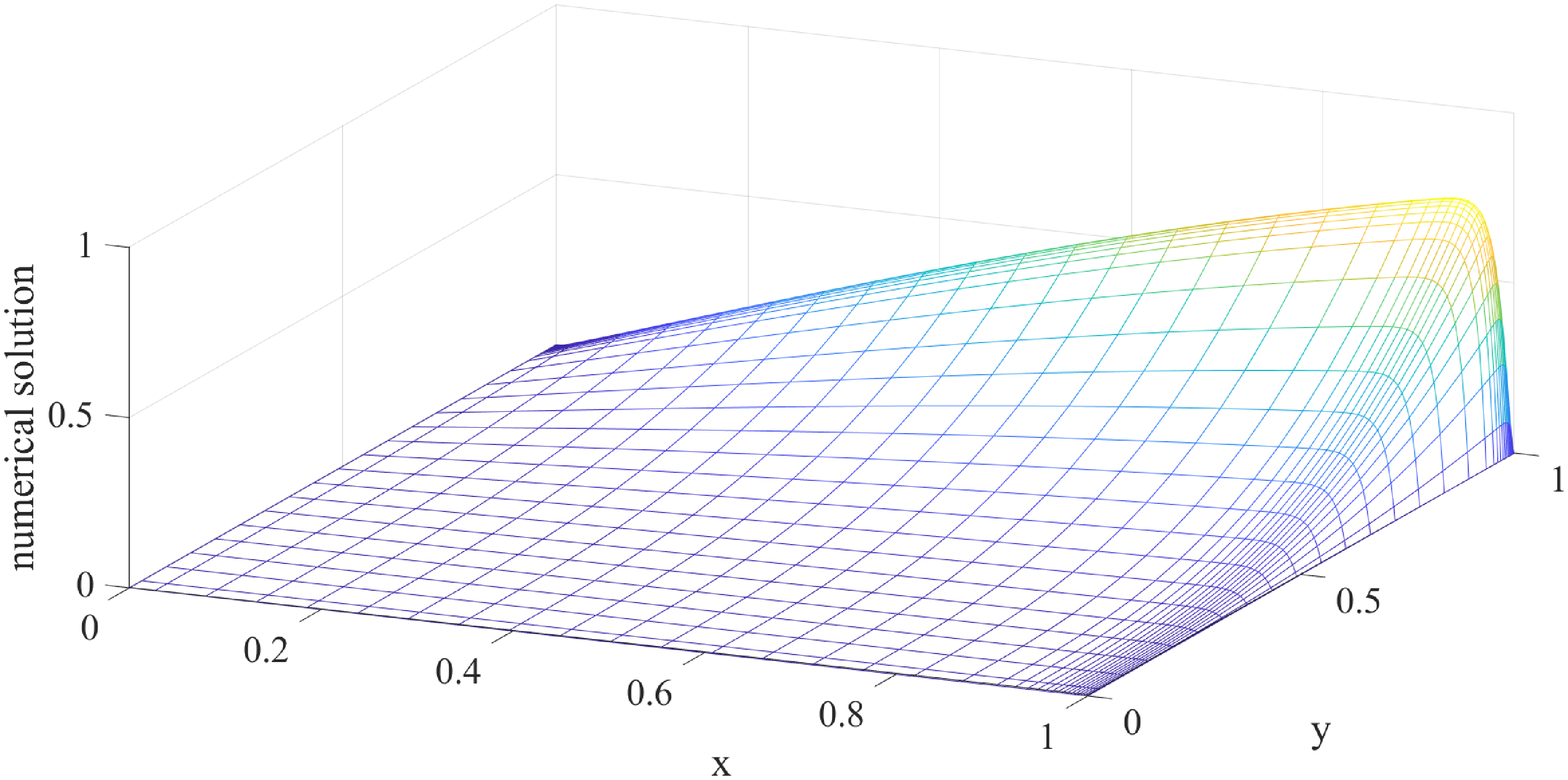}}
		\centerline{\includegraphics[width=2.6in,height=2.4in]{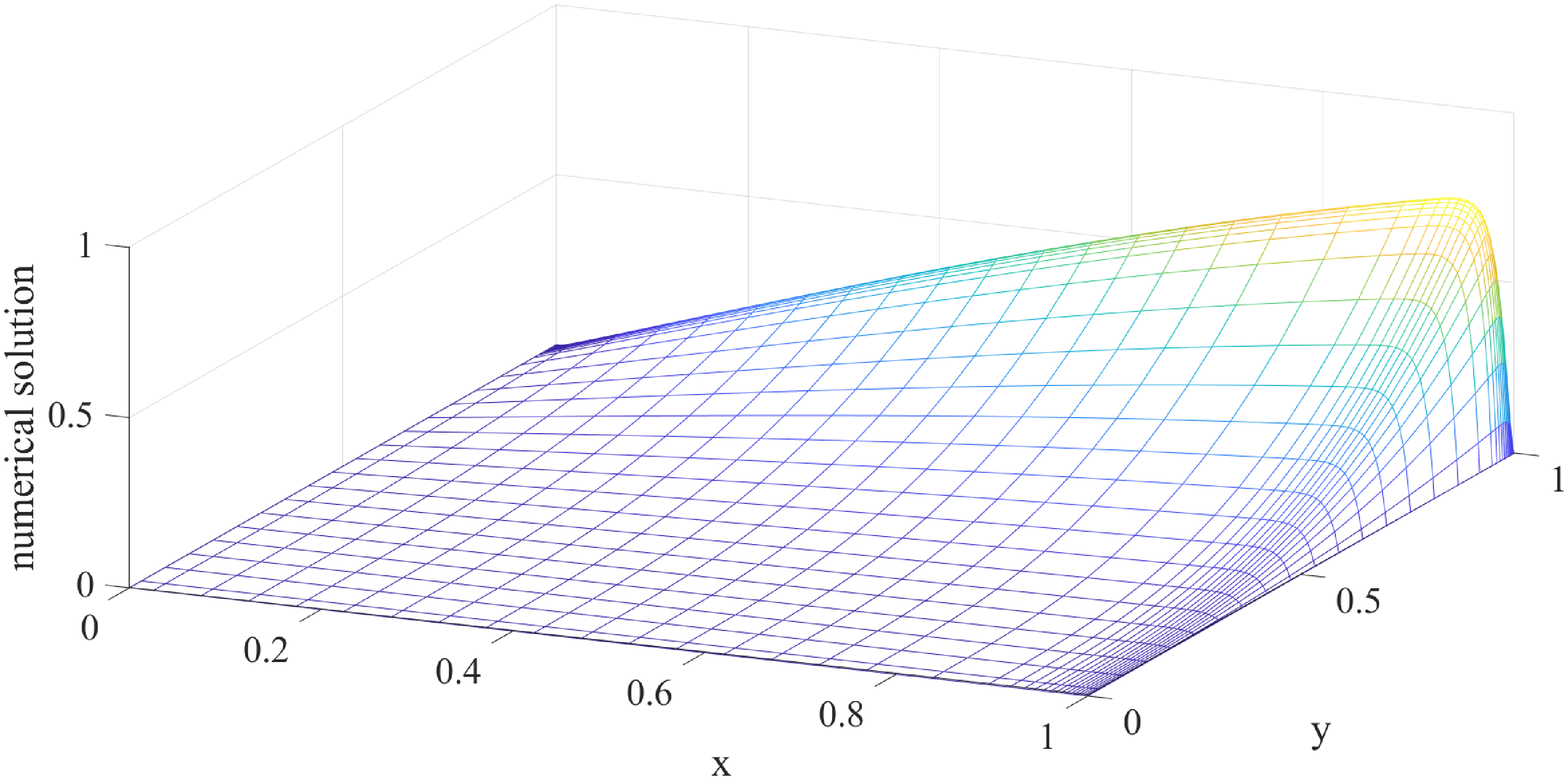}}
	\end{minipage}
	\begin{minipage}{0.49\linewidth}
		\centerline{\includegraphics[width=2.6in,height=2.4in]{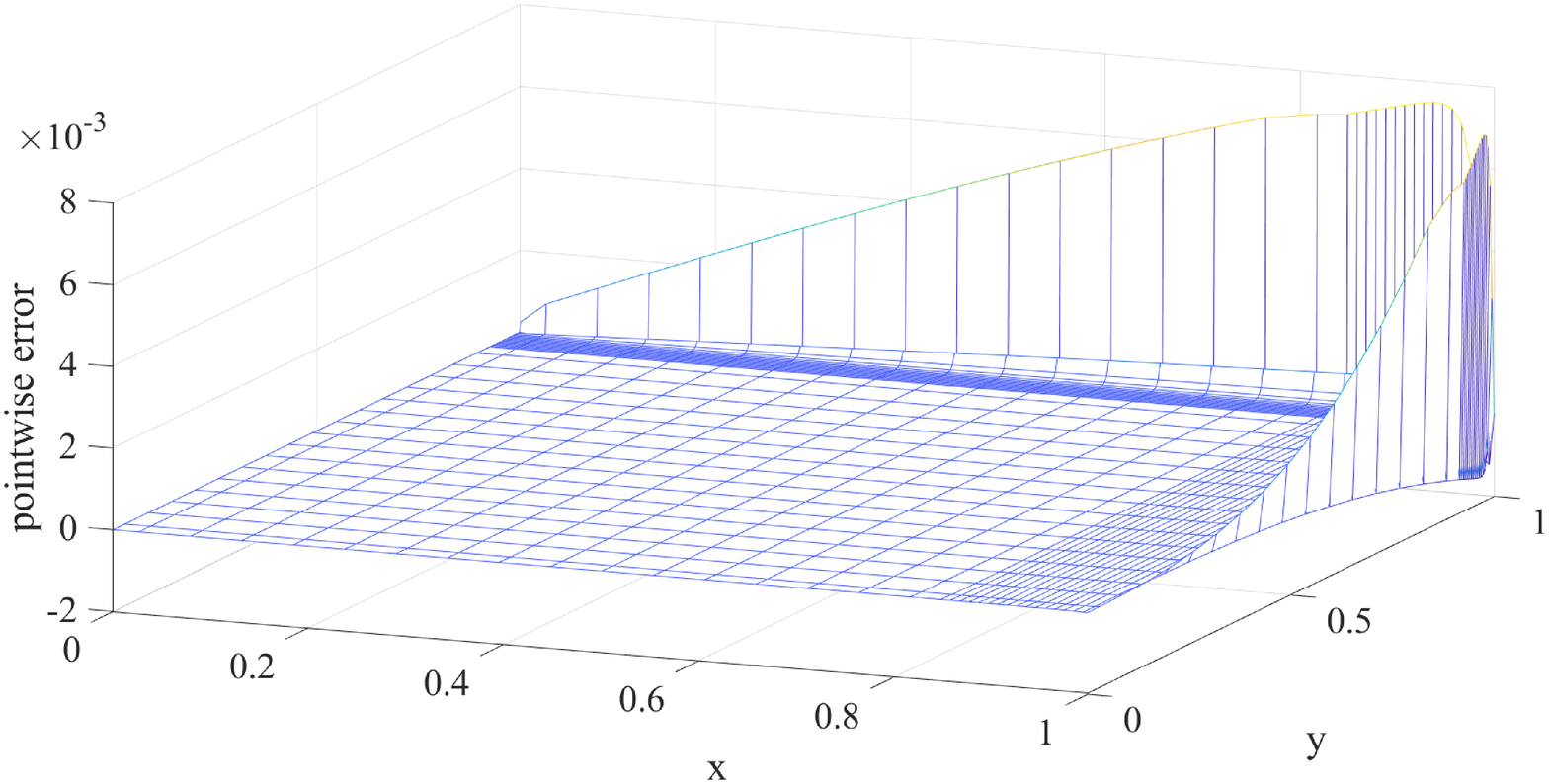}}
		\centerline{\includegraphics[width=2.6in,height=2.4in]{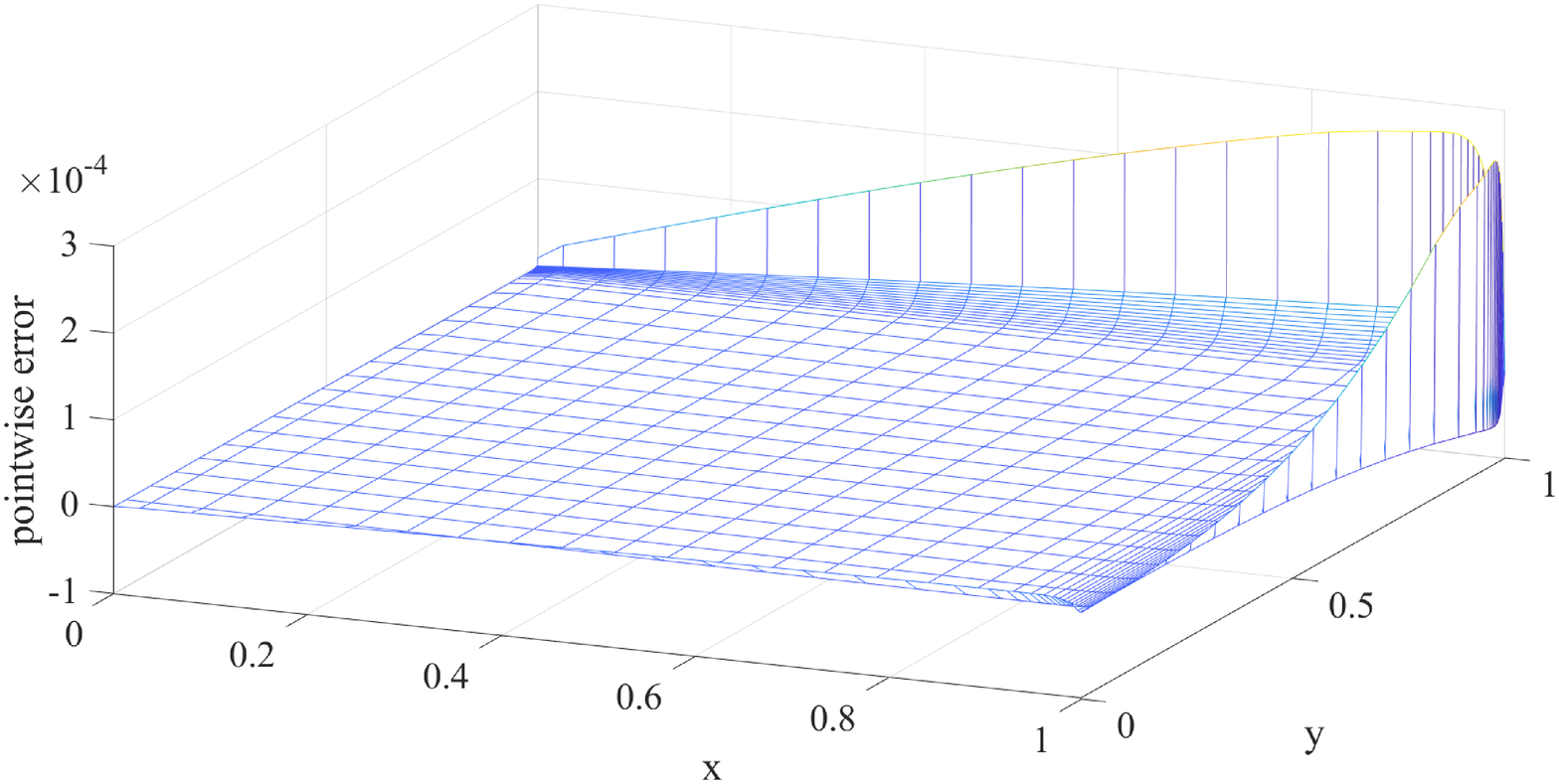}}
		\centerline{\includegraphics[width=2.6in,height=2.4in]{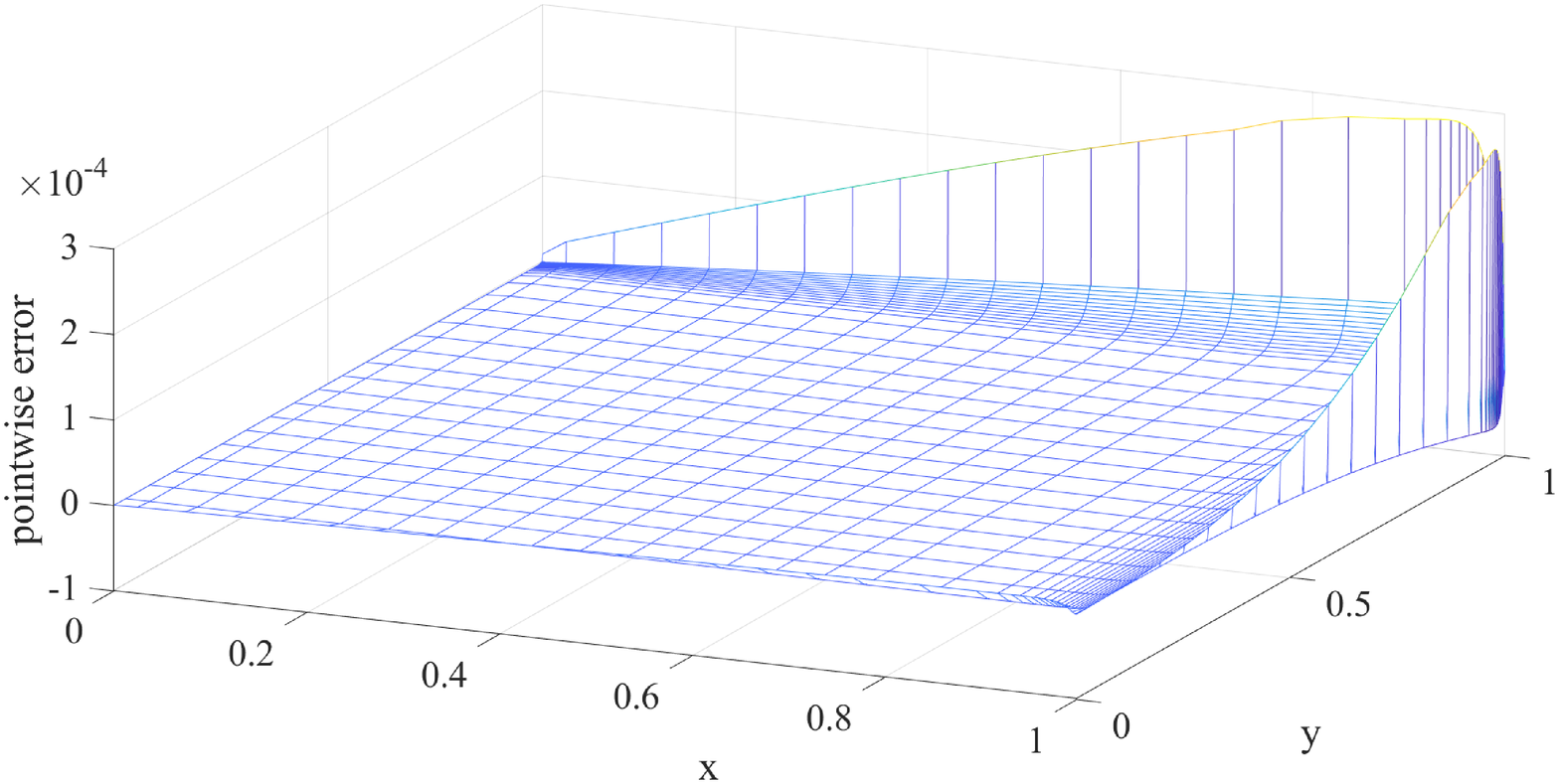}}
	\end{minipage}
	\caption{Numerical solution $U$ (left) and error $u-U$ (right) on three 
		$32\times32$ layer-adapted meshes. Top: S-mesh; Middle: BS-mesh; Bottom: B-mesh. Here $\varepsilon=10^{-2}$ and $k=2$.}
	\label{figure:1}
\end{figure}

\begin{table}[ht]
	\small
	\centering
	\caption{S-mesh.}
	\smallskip
	\label{table:error:S}
	\begin{tabular}{cccccccc}
		\toprule
		& $N$ & $\lnorm{\bm{w}-\wN}$ & $r_S$-rate 
		& $\enorm{\Pi \bm{w}-\wN}$ & $r_S$-rate  
		& $\enorm{\bm{w}-\wN}$ & $r_S$-rate
		\\
		\midrule
		$\mathcal{P}^1$& 16	&3.0199E-02	&-	&8.3413E-02	&-	&6.7252E-02	&-  \\
		& 32	&1.2885E-02	&1.8122 	&3.9597E-02	&1.5852 	&3.4133E-02	&1.4429   \\
		& 64	&4.8976E-03	&1.8936 	&1.6364E-02	&1.7299 	&1.6102E-02	&1.4708   \\
		& 128	&1.7197E-03	&1.9418 	&6.0899E-03	&1.8339 	&7.2263E-03	&1.4865   \\
		& 256	&5.7159E-04	&1.9683 	&2.1008E-03	&1.9018 	&3.1321E-03	&1.4939   \\
		$\mathcal{P}^2$& 16	&5.7757E-03	&-	&1.7540E-02	&-	&1.3364E-02	&-  \\
		& 32	&1.5824E-03	&2.7546 	&5.5245E-03	&2.4581 	&4.4268E-03	&2.3508   \\
		& 64	&3.6396E-04	&2.8771 	&1.4253E-03	&2.6522 	&1.2753E-03	&2.4363   \\
		& 128	&7.4631E-05	&2.9397 	&3.1717E-04	&2.7880 	&3.3600E-04	&2.4746   \\
		& 256	&1.4161E-05	&2.9700 	&6.3447E-05	&2.8756 	&8.3405E-05	&2.4899   \\
		\bottomrule
	\end{tabular}
\end{table}

\begin{table}[ht]
	\small
	\centering
	\caption{BS-mesh.}
	\smallskip
	\label{table:error:BS}
	\begin{tabular}{cccccccc}
		\toprule
		& $N$ & $\lnorm{\bm{w}-\wN}$ & $r_2$-rate & $\enorm{\Pi \bm{w}-\wN}$ & $r_2$-rate  
		& $\enorm{\bm{w}-\wN}$ & $r_2$-rate
		\\
		\midrule
		
		$\mathcal{P}^1$
		& 16  & 7.4770e-03 & -      & 1.5422e-02 & -      & 2.4479e-02 & -      \\
		& 32  & 2.0492e-03 & 1.8674 & 4.2764e-03 & 1.8506 & 8.9815e-03 & 1.4465 \\
		& 64  & 5.3808e-04 & 1.9292 & 1.1260e-03 & 1.9251 & 3.2362e-03 & 1.4727 \\
		& 128 & 1.3802e-04 & 1.9630 & 2.8893e-04 & 1.9625 & 1.1552e-03 & 1.4861 \\
		& 256 & 3.4963e-05 & 1.9809 & 7.3172e-05 & 1.9814 & 4.1043e-04 & 1.4930 \\
		
		$\mathcal{P}^2$
		& 16  & 5.3110e-04 & -      & 1.2764e-03 & -      & 1.6682e-03 & -      \\
		& 32  & 7.5788e-05 & 2.8089 & 1.8291e-04 & 2.8029 & 3.2007e-04 & 2.3818 \\
		& 64  & 1.0138e-05 & 2.9022 & 2.4466e-05 & 2.9023 & 5.8928e-05 & 2.4414 \\
		& 128 & 1.3120e-06 & 2.9499 & 3.1666e-06 & 2.9498 & 1.0630e-05 & 2.4708 \\
		& 256 & 1.6699e-07 & 2.9739 & 4.0197e-07 & 2.9778 & 1.9006e-06 & 2.4836 \\
		\bottomrule
	\end{tabular}
\end{table}

\begin{table}[ht]
	\small
	\centering
	\caption{B-mesh.}
	\smallskip
	\label{table:error:B}
	\begin{tabular}{cccccccc}
		\toprule
		& $N$ & $\lnorm{\bm{w}-\wN}$ & $r_2$-rate 
		& $\enorm{\Pi \bm{w}-\wN}$ & $r_2$-rate  
		& $\enorm{\bm{w}-\wN}$ & $r_2$-rate
		\\
		\midrule
		$\mathcal{P}^1$
		& 16  & 8.2628e-03 & -      & 1.7412e-02 & -      & 2.6065e-02 & -      \\
		& 32  & 2.1582e-03 & 1.9368 & 4.5461e-03 & 1.9374 & 9.2693e-03 & 1.4916 \\
		& 64  & 5.5248e-04 & 1.9658 & 1.1612e-03 & 1.9690 & 3.2879e-03 & 1.4953 \\
		& 128 & 1.3987e-04 & 1.9818 & 2.9339e-04 & 1.9847 & 1.1645e-03 & 1.4975 \\
		& 256 & 3.5198e-05 & 1.9905 & 7.3738e-05 & 1.9923 & 4.1206e-04 & 1.4987 \\
		$\mathcal{P}^2$
		& 16  & 6.4042e-04 & -      & 1.5414e-03 & -      & 1.9524e-03 & -      \\
		& 32  & 8.3189e-05 & 2.9446 & 2.0092e-04 & 2.9396 & 3.4568e-04 & 2.4977 \\
		& 64  & 1.0622e-05 & 2.9694 & 2.5647e-05 & 2.9698 & 6.1218e-05 & 2.4974 \\
		& 128 & 1.3427e-06 & 2.9838 & 3.2452e-06 & 2.9824 & 1.0834e-05 & 2.4984 \\
		& 256 & 1.6887e-07 & 2.9912 & 4.0274e-07 & 3.0104 & 1.9221e-06 & 2.4948 \\
		\bottomrule
	\end{tabular}
\end{table}

\begin{table}[ht]
	\small
	\centering
	\caption{S-mesh.}
	\smallskip
	\label{table:robust:S}
	\begin{tabular}{cccc}
		\toprule
		$\varepsilon$ & $\lnorm{\bm{w}-\wN}$  & $\enorm{\Pi \bm{w}-\wN}$   
		& $\enorm{\bm{w}-\wN}$ 
		\\
		\midrule
		$10^{-3}$	&7.4924E-05		&3.1508E-04		&3.3673E-04	  \\
		$10^{-4}$	&7.4659E-05		&3.1696E-04		&3.3607E-04	  \\
		$10^{-5}$	&7.4633E-05		&3.1715E-04		&3.3600E-04	  \\
		$10^{-6}$	&7.4630E-05		&3.1717E-04		&3.3599E-04	  \\
		$10^{-7}$	&7.4630E-05		&3.1717E-04		&3.3599E-04	  \\
		$10^{-8}$	&7.4631E-05		&3.1717E-04		&3.3600E-04	  \\
		$10^{-9}$	&7.4642E-05		&3.1729E-04		&3.3596E-04	  \\
		$10^{-10}$	&7.4654E-05		&3.1722E-04		&3.3560E-04	  \\
		\bottomrule
	\end{tabular}
\end{table}

\begin{table}[ht]
	\small
	\centering
	\caption{BS-mesh.}
	\smallskip
	\label{table:robust:BS}
	\begin{tabular}{cccc}
		\toprule
		$\varepsilon$ & $\lnorm{\bm{w}-\wN}$  & $\enorm{\Pi \bm{w}-\wN}$   
		& $\enorm{\bm{w}-\wN}$ 
		\\
		\midrule
		$10^{-3}$	&1.3141e-06		&3.1511e-06		&1.0630e-05   \\
		$10^{-4}$	&1.3121e-06		&3.1626e-06		&1.0631e-05	  \\
		$10^{-5}$	&1.3119e-06		&3.1637e-06		&1.0631e-05	  \\
		$10^{-6}$	&1.3119e-06		&3.1638e-06		&1.0631e-05   \\
		$10^{-7}$	&1.3119e-06		&3.1640e-06		&1.0631e-05	  \\
		$10^{-8}$	&1.3120e-06		&3.1666e-06		&1.0630e-05   \\
		$10^{-9}$	&1.3089e-06		&3.0693e-06 	&1.0707e-05   \\
		$10^{-10}$  &1.4170e-06 	&3.3061e-06		&1.0862e-05	  \\
		\bottomrule
	\end{tabular}
\end{table}

\begin{table}[ht]
	\small
	\centering
	\caption{B-mesh.}
	\smallskip
	\label{table:robust:B}
	\begin{tabular}{cccc}
		\toprule
		$\varepsilon$ & $\lnorm{\bm{w}-\wN}$  & $\enorm{\Pi \bm{w}-\wN}$   
		& $\enorm{\bm{w}-\wN}$ 
		\\
		\midrule
		$10^{-3}$	&1.3410e-06		&3.2052e-06		&1.0804e-05  \\
		$10^{-4}$	&1.3426e-06		&3.2342e-06		&1.0831e-05	  \\
		$10^{-5}$	&1.3427e-06		&3.2383e-06		&1.0834e-05	  \\
		$10^{-6}$	&1.3427e-06		&3.2388e-06		&1.0835e-05  \\
		$10^{-7}$	&1.3428e-06		&3.2394e-06		&1.0835e-05	  \\
		$10^{-8}$	&1.3427e-06		&3.2452e-06		&1.0834e-05  \\
		$10^{-9}$	&1.3441e-06		&3.3431e-06 	&1.0820e-05	  \\
		$10^{-10}$	&1.4046e-06		&4.9222e-06		&1.0738e-05	  \\
		\bottomrule
	\end{tabular}
\end{table}
\section{Concluding remarks}
\label{sec:conclusion}

In this paper we considered a singularly perturbed convection-diffusion problem posed on the unit square and derived a supercloseness result 
for the energy-norm error between the numerical solution computed by the LDG method
on three typical layer-adapted meshes and the local Gauss-Radau projection of the exact solution into the finite element space.
As a byproduct,
we got an (almost) optimal-order $L^2$-error estimate for the LDG solution.
These results rely on some new sharp estimates 
for the error of the Gauss-Radau projection that may be of independent interest.
Numerical experiments show that our theoretical bounds are sharp
(sometimes up to a logarithmic factor).
In  future work we aim to extend these results  
to singularly perturbed problems whose solutions have characteristic boundary layers, where superconvergence in the energy and balanced norms will be investigated.

%


\bibliography{ChengJiangStynes}

\begin{thebibliography}{10}

\bibitem{Castillo2002}
Paul Castillo, Bernardo Cockburn, Dominik Sch\"{o}tzau, and Christoph Schwab.
\newblock Optimal a priori error estimates for the {$hp$}-version of the local
  discontinuous {G}alerkin method for convection-diffusion problems.
\newblock {\em Math. Comp.}, 71(238):455--478, 2002.

\bibitem{Cheng2021:Calcolo}
Yao Cheng and Yanjie Mei.
\newblock Analysis of generalised alternating local discontinuous {G}alerkin
  method on layer-adapted mesh for singularly perturbed problems.
\newblock {\em Calcolo}, 58(4):Paper No. 52, 36, 2021.

\bibitem{Cheng2020}
Yao Cheng, Yanjie Mei, and Hans-G\"{o}rg Roos.
\newblock The local discontinuous {G}alerkin method on layer-adapted meshes for
  time-dependent singularly perturbed convection-diffusion problems.
\newblock {\em Comput. Math. Appl.}, 117:245--256, 2022.

\bibitem{CYWL22}
Yao Cheng, Li~Yan, Xuesong Wang, and Yanhua Liu.
\newblock Optimal maximum-norm estimate of the {LDG} method for singularly
  perturbed convection-diffusion problem.
\newblock {\em Appl. Math. Lett.}, 128:Paper No. 107947, 11, 2022.

\bibitem{Cockburn:Dong:2007}
Bernardo Cockburn and Bo~Dong.
\newblock An analysis of the minimal dissipation local discontinuous {G}alerkin
  method for convection-diffusion problems.
\newblock {\em J. Sci. Comput.}, 32(2):233--262, 2007.

\bibitem{Cockburn2001}
Bernardo Cockburn, Guido Kanschat, Ilaria Perugia, and Dominik Sch\"{o}tzau.
\newblock Superconvergence of the local discontinuous {G}alerkin method for
  elliptic problems on {C}artesian grids.
\newblock {\em SIAM J. Numer. Anal.}, 39(1):264--285, 2001.

\bibitem{Cockburn:Shu:LDG}
Bernardo Cockburn and Chi-Wang Shu.
\newblock The local discontinuous {G}alerkin method for time-dependent
  convection-diffusion systems.
\newblock {\em SIAM J. Numer. Anal.}, 35(6):2440--2463, 1998.

\bibitem{Cockburn:Shu:2001}
Bernardo Cockburn and Chi-Wang Shu.
\newblock Runge-{K}utta discontinuous {G}alerkin methods for
  convection-dominated problems.
\newblock {\em J. Sci. Comput.}, 16(3):173--261, 2001.

\bibitem{JKN18}
Volker John, Petr Knobloch, and Julia Novo.
\newblock Finite elements for scalar convection-dominated equations and
  incompressible flow problems: a never ending story?
\newblock {\em Comput. Vis. Sci.}, 19(5-6):47--63, 2018.

\bibitem{Linss10}
Torsten Lin\ss.
\newblock {\em Layer-adapted meshes for reaction-convection-diffusion
  problems}, volume 1985 of {\em Lecture Notes in Mathematics}.
\newblock Springer-Verlag, Berlin, 2010.

\bibitem{Miller2012}
J.~J.~H. Miller, E.~O'Riordan, and G.~I. Shishkin.
\newblock {\em Fitted numerical methods for singular perturbation problems}.
\newblock World Scientific Publishing Co. Pte. Ltd., Hackensack, NJ, revised
  edition, 2012.
\newblock Error estimates in the maximum norm for linear problems in one and
  two dimensions.

\bibitem{Reed1973}
W.H. Reed and T.R. Hill.
\newblock Triangular mesh methods for the neutron transport equation.
\newblock Technical Report LA-UR-73-479, Los Alamos Scientific Laboratory, Los
  Alamos, 1973.

\bibitem{Roos2008}
Hans-G\"{o}rg Roos, Martin Stynes, and Lutz Tobiska.
\newblock {\em Robust numerical methods for singularly perturbed differential
  equations}, volume~24 of {\em Springer Series in Computational Mathematics}.
\newblock Springer-Verlag, Berlin, second edition, 2008.
\newblock Convection-diffusion-reaction and flow problems.

\bibitem{Roos2003}
Hans-G{\"o}rg Roos and Helena Zarin.
\newblock The discontinuous galerkin finite element method for singularly
  perturbed problems.
\newblock In Eberhard B{\"a}nsch, editor, {\em Challenges in Scientific
  Computing - CISC 2002}, pages 246--267, Berlin, Heidelberg, 2003. Springer
  Berlin Heidelberg.

\bibitem{StySty18}
Martin Stynes and David Stynes.
\newblock {\em Convection-diffusion problems}, volume 196 of {\em Graduate
  Studies in Mathematics}.
\newblock American Mathematical Society, Providence, RI; Atlantic Association
  for Research in the Mathematical Sciences (AARMS), Halifax, NS, 2018.
\newblock An introduction to their analysis and numerical solution.

\bibitem{Wang2015}
Haijin Wang, Shiping Wang, Qiang Zhang, and Chi-Wang Shu.
\newblock Local discontinuous {G}alerkin methods with implicit-explicit
  time-marching for multi-dimensional convection-diffusion problems.
\newblock {\em ESAIM Math. Model. Numer. Anal.}, 50(4):1083--1105, 2016.

\bibitem{Xie2010MC}
Ziqing Xie and Zhimin Zhang.
\newblock Uniform superconvergence analysis of the discontinuous {G}alerkin
  method for a singularly perturbed problem in 1-{D}.
\newblock {\em Math. Comp.}, 79(269):35--45, 2010.

\bibitem{Xie2009JCM}
Ziqing Xie, Zuozheng Zhang, and Zhimin Zhang.
\newblock A numerical study of uniform superconvergence of {LDG} method for
  solving singularly perturbed problems.
\newblock {\em J. Comput. Math.}, 27(2-3):280--298, 2009.

\bibitem{Zarin2014}
Helena Zarin.
\newblock On discontinuous {G}alerkin finite element method for singularly
  perturbed delay differential equations.
\newblock {\em Appl. Math. Lett.}, 38:27--32, 2014.

\bibitem{Zhu:2018}
Huiqing Zhu and Fatih Celiker.
\newblock Nodal superconvergence of the local discontinuous {G}alerkin method
  for singularly perturbed problems.
\newblock {\em J. Comput. Appl. Math.}, 330:95--116, 2018.

\bibitem{Zhu2013}
Huiqing Zhu and Zhimin Zhang.
\newblock Convergence analysis of the {LDG} method applied to singularly
  perturbed problems.
\newblock {\em Numer. Methods Partial Differential Equations}, 29(2):396--421,
  2013.

\bibitem{Zhu:2dMC}
Huiqing Zhu and Zhimin Zhang.
\newblock Uniform convergence of the {LDG} method for a singularly perturbed
  problem with the exponential boundary layer.
\newblock {\em Math. Comp.}, 83(286):635--663, 2014.

\end{thebibliography}
\bibliographystyle{amsplain}

\end{document}